\documentclass[a4paper,12pt]{article}
\usepackage{hyperref}
\usepackage[T1]{fontenc}
\usepackage[utf8]{inputenc}

\usepackage[english]{babel}
\usepackage{mathtools}
\usepackage{amsfonts}
\usepackage{color}
\usepackage{amsmath}
\usepackage{amsthm}
\usepackage{amssymb}
\usepackage{mathrsfs}
\usepackage{amstext}
\usepackage{float}

\usepackage{todonotes}

\usepackage[normalem]{ulem}

\usepackage{graphicx}
\usepackage{tikz}
\usepackage{subfig}
\usepackage{pgfplots}
\usepgfplotslibrary{fillbetween}

\usepackage{authblk}
\usepackage{stmaryrd}
\usepackage{makeidx}
\makeindex

\numberwithin{equation}{section}

\theoremstyle{definition}
\newtheorem{thm}{Theorem}[section]
\newtheorem{rem}[thm]{Remark}
\newtheorem{defi}[thm]{Definition}
\newtheorem{lem}[thm]{Lemma}
\newtheorem{cor}[thm]{Corollary}
\newtheorem{prop}[thm]{Proposition}

\newtheorem{q}[thm]{Question}

\newtheorem*{thm*}{Theorem}
\newtheorem*{rem*}{Remark}
\newtheorem*{folg*}{Folgerung}
\newtheorem*{examples*}{Beispiele}
\newtheorem*{ex*}{Beispiel}
\newtheorem*{lem*}{Lemma}
\newtheorem*{prop*}{Proposition}
\newtheorem*{defi*}{Definition}
\newtheorem*{exercise*}{Übung}
\newtheorem*{conj*}{Conjecture}
\newtheorem*{q*}{Question}

\usepackage{cleveref}

\newcommand{\ba}{\mathbf{a}}

\newcommand{\bm}{\mathbf{m}}

\newcommand{\bx}{\mathbf{x}}
\newcommand{\by}{\mathbf{y}}

\newcommand{\bbN}{\mathbb{N}}

\newcommand{\bbR}{\mathbb{R}}

\newcommand{\bbT}{\mathbb{T}}

\newcommand{\bbZ}{\mathbb{Z}}

\newcommand{\ra}{\rightarrow}

\newcommand{\sm}{\setminus}
\newcommand{\sbse}{\subseteq}

\newcommand{\dv}{\operatorname{dv}}

\newcommand{\iu}{\mathrm{i}}
\newcommand{\id}{\operatorname{id}}

\newcommand{\supp}{\operatorname{supp}}

\newcommand{\TBC}{\textcolor{red}{\textbf{TO BE CONTINUED}}}

\newcommand{\dx}{\text{d}}

\allowdisplaybreaks

\author{Dmitriy Bilyk, Nicolas Nagel, Ian Ruohoniemi}
\title{Minimizing point configurations for tensor product energies on the torus}
\date{}

\pgfplotsset{yticklabel style={text width=3em,align=right}}

\begin{document}
	
	\maketitle
	
    
\begin{abstract}
 We study point configurations on the torus $\mathbb T^d$  that  minimize interaction energies with tensor product structure which arise naturally  in the context of discrepancy theory and  quasi-Monte Carlo integration. Permutation sets on $\mathbb T^2$  and Latin hypercube sets in higher dimensions (i.e. sets whose projections onto coordinate axes are equispaced points)  are natural candidates to be energy minimizers. We show that such point configurations that have only one distance in the vector sense minimize the energy for a wide range of potentials, in other words, such sets satisfy a tensor product version of universal optimality. This applies, in particular, to three- and five-point Fibonacci lattices. We also  characterize all lattices with this property and exhibit some non-lattice sets of this type.  In addition, we obtain several further structural results about global and local minimizers of tensor product energies. 
 
\end{abstract}

\section{Introduction}

\subsection{Energies with product structure}

Denote by $\bbT = \bbR/\bbZ$ the unit  torus, where we will identify elements in $\bbT$ with those in the half-open interval $[0, 1)$. In the present paper we explore the properties of minimizers of  interaction energies on $\mathbb T^d$ with kernels that have tensor product structure.  To be more precise, consider the kernel  $F: \mathbb T^d  \rightarrow [0, \infty)$ of the form 
\begin{equation}\label{e.f}
F(t_1, \dots, t_d)  = f(t_1) \dots f(t_d), 
\end{equation}
where 
$f: \bbT \ra [0, \infty)$. Throughout the paper we will assume that $f$, viewed  as   a non-negative  $1$-periodic function on $\bbR$, is \emph{even and continuous} on $(0,1)$. In most situations $f$ will actually be assumed to be continuous on the whole real line, but some  results also apply when $f$ is singular at $0$ in the sense of $\lim_{t \ra 0}f(t) = + \infty$. We will refer to both $F$ and $f$ as \emph{potentials} throughout the paper.

For a (multi)set of $N$ points $X = \{ \bx^0, \bx^1, \dots, \bx^{N-1} \} \subset \mathbb T^d$ with components $\bx^k = (x^k_1, \dots, x^k_d)$, we define the \emph{tensor product energy} as 
\begin{equation}\label{e.energy}
E_F(X) 
= \sum_{\substack{k, \ell=0 \\ k \neq \ell}}^{N-1}F({\bf{x}}^k -  {\bf{x}}^\ell) = \sum_{\substack{k, \ell=0 \\ k \neq \ell}}^{N-1} \prod_{i=1}^d  f(x_i^k - x_i^\ell).
\end{equation}
This quantity can be intuitively viewed as the total potential energy of $N$ particles in positions ${\bf x}^0, \bx^1, \dots, {\bf x}^{N-1} $ which interact according to the potential given by $F$. One is then naturally interested in minimizing the energy \eqref{e.energy} over all $N$-point configurations and finding minimizing configurations of points, i.e.  equilibrium distributions. We shall call such configurations \emph{energy minimizers}. While extensive literature  exists on energy minimization in various domains in the case where the kernel $F({\bf{x}}, {\bf{y}})$ depends on the Euclidean (or another) distance between ${\bf{x}}$ and  ${\bf{y}}$ (see for example the book \cite{BHS}), kernels with tensor product structures have not received the same attention \cite{HO16, Nag25b}. 

The condition that $f$ is even, or equivalently $f(t) = f (1-t)$, means that $f (x - y)$ depends only on the \emph{geodesic distance}
\begin{align} \label{eq:geod_dist}
    \delta(x, y) = \min\{|x-y|, 1-|x-y|\}
\end{align}
between $x$ and $y$ in $\mathbb T$. 
For most potentials of interest, $f$ is decreasing on $(0,1/2]$, as this corresponds to the interaction being repulsive in each coordinate. 

 We remark that  the energy  \eqref{e.energy} of a configuration is invariant under symmetries of the torus (translations, reflections, and permutations of coordinates) and all uniqueness statements, e.g. in Definition \ref{def:perm_set} or Theorem \ref{thm:unique_dist_lat}, are to be understood modulo these symmetries.  In particular, since  the energy \eqref{e.energy} is translation invariant, we may assume without loss of generality  that ${\bf x}^0 = (0, \dots, 0)$.

Energies of the type \eqref{e.energy} with tensor product kernels \eqref{e.f} naturally arise in several fields such as discrepancy theory and quasi-Monte Carlo (QMC) integration. Warnock's formula \cite{HKP21, HO16, War72} states that minimizing the periodic $L^2$-discrepancy of $N$ points in $\mathbb T^d$ is equivalent to minimizing the  energy \eqref{e.energy} with the kernel \eqref{e.f}, where $f(t) = \frac12 - t + t^2$ on $[0,1)$. It is also known \cite{HO16, Nag25b}  that minimizing the worst-case integration error of the quasi-Monte Carlo rule over a certain Sobolev space corresponds to the energy minimization problem with the potential
\begin{align} \label{eq:qmc_potential}
    f_p(t) = 1 + \frac{p}{2} \left( \frac16 - t + t^2\right),
\end{align}
where $p > 0$ is a fixed parameter (notice that $f_6= 3f$, where $f$ is the aforementioned potential from Warnock's formula). Similar statements hold for other reproducing kernel Hilbert spaces (see the Appendix §\ref{s.potentials} and \cite{BD09, BT03} or \cite[§9]{NW10}). 
The relation to discrepancy theory provides possible candidates for  energy minimizers, namely low-discrepancy sets. 

\subsection{Dimension $d=2$: Permutation sets and Fibonacci lattices} 

Some of the most well-known low-discrepancy point sets in $\bbT^2$ include the digit-reversing \emph{van der Corput sets} when $N = 2^n$  (see e.g. \cite{HKP21} for a definition) and the \emph{Fibonacci lattices} which  will keep appearing in  this paper: if $N = F_j$ is a Fibonacci number, the $N$-point Fibonacci lattice is defined by
\begin{align} \label{eq:fib_lat}
	\Phi_j \coloneqq \left\{\left(\frac{k}{F_j}, \left\{\frac{k F_{j-1}}{F_j}\right\}\right): k=0, 1, \dots, F_j-1\right\},
\end{align}
where $\{\cdot\}$ denotes the fractional part.  This set can be thought of as a rational approximation to the famous Kronecker set $\big\{  (k/N, \{ k \varphi \} ): k \in \bbN_0 \big\}$, where $\varphi$ is the golden section. Fibonacci lattices have been studied extensively in the context of discrepancy theory and numerical integration and perform strongly in these settings \cite{BTY12-2, BTY12-1}.

Both of the aforementioned examples, the van der Corput set and the Fibonacci lattice, have important structural properties. First, these are $N$-point configurations that are subsets of the $(N\times N)$-grid
$$
\left\{ \left( \frac{k}{N}, \frac{\ell}{N} \right): \, k, \ell = 0,1,\dots, N-1 \right\}
$$
in $ [0,1)^2 \simeq \mathbb T^2$. We shall call such sets  {\it grid sets}. More importantly these examples are of the form 
\begin{equation}\label{e.permutation}
X_\sigma = \{ {\bf x}^k \}_{k=0}^{N-1} = \left\{\left(\frac{k}{N}, \frac{\sigma (k) }{N}\right): k=0, 1, \dots, N-1\right\},
\end{equation}
where $\sigma$ is a permutation of the set $\{0,1,\dots, N-1\}$  (if we assume ${\bf x}^0 = (0,0)$, that is we fix $\sigma (0) = 0$, we can view $\sigma $ just as easily as a permutation of $\{ 1,\dots, N-1\}$).

\begin{defi}\label{def:perm_set}
We shall call configurations $X \subset \mathbb T^2$ which (up to translations) have the form \eqref{e.permutation}  {\emph{permutation sets}}.  
\end{defi} 

A special place among the permutation sets is taken by  {\it integration lattices} (which we shall refer to simply as {\it lattices}), i.e. sets of the form 
\begin{equation}\label{e.lattice}
\left\{\left(\frac{k}{N}, \left\{\frac{hk }{N}\right\}\right): k=0, 1, \dots, N-1\right\},
\end{equation}
where $h$ and $N$ are relatively prime integers. The condition that $h$ and $N$ are relatively prime guarantees that \eqref{e.lattice} is indeed a permutation set. Notice that such sets are restrictions to $[0,1)^2$ of lattices in $\mathbb R^2$ generated by the vectors $(1/N, h/N)$ and $(0,1)$, which are $1$-periodic in both coordinate directions. Fibonacci lattices \eqref{eq:fib_lat} are particular examples of such sets.

We note that permutation sets are those sets whose projections onto both $x$- and $y$-axes are $N$ equispaced points, i.e. they are grid sets with exactly one point in each row and each column.  This interpretation suggests that, since many energies in dimension $d=1$ are minimized by equispaced points (see \S \ref{sec:equispaced}), it is natural to expect that permutation sets minimize tensor product energies on $\mathbb T^2$ for large classes of potentials. We then have the following question: {\emph{which conditions on the potential $f$ guarantee that  minimizers of the tensor product energy \eqref{e.energy} are permutation sets?}} This question admits a variety of versions and interpretations: e.g., when does this hold for all $N$, for specific values of $N$, or for all $N$ large enough? When are energies minimized by lattices? Does it hold if sets are restricted to the $(N\times N)$-grid? 

Based on the  particularly important role played by the Fibonacci lattices in low discrepancy constructions and on ample numerical evidence \cite{HO16}, we conjecture that whenever $N$ is a Fibonacci number,  \emph{the Fibonacci lattice minimizes the energy for a broad class of potentials $f$}. This conjectured phenomenon (the same configuration minimizing a large variety of energies) is a natural analogue of the notion of {\emph{universal optimality}}  that has been studied on the sphere $\mathbb S^d$ \cite{CK07} and in the plane \cite{CKMRV22}, and is known to be very rare.

In the end, we remark that the condition that $f$ is logarithmically convex is very natural as it is necessary and sufficient for the tensor product $F$ to be convex as a function of $d$ variables \cite{SS12b}. 

\subsection{Higher dimensions $d\ge 3$} 
In dimensions $d \geq 3$ it becomes significantly harder to construct good point sets. To the knowledge of the authors, no study of global optimizers of $L_2$-discrepancy or quasi-Monte Carlo worst case errors in dimensions $d \geq 3$ has been carried out so far and only asymptotic results are known \cite{CS02, GSY16a}. A straightforward extension of the concept of permutation sets in higher dimensions to sets on the $N^d$-grid $\{k/N: k=0, 1, \dots, N-1\}^d$ leads to Latin (hyper)cubes.

\begin{defi} \label{def:lat_hypcube}
    A \emph{Latin hypercube set} is defined by permutations $\sigma_2, \dots, \sigma_d$ of the set $\{0, 1, \dots, N-1\}$ via
    $$
    X_{\sigma_2, \dots, \sigma_d} = \left\{\left(\frac{k}{N}, \frac{\sigma_2 (k) }{N}, \dots, \frac{\sigma_d(k)}N\right): k=0, 1, \dots, N-1\right\}.
    $$
\end{defi}
In other words, it is a set whose projections on all coordinate axes are equispaced points.  Thus a {\emph{permutation set}} is a Latin hypercube set with $d=2$. 
Analogously to $\mathbb T^2$, lattices in $\bbT^d$ of the form 
$$\left\{\left(\frac{k}{N}, \left\{ \frac{h_2 k }{N}\right\} , \dots, \left\{ \frac{h_{d} k}N \right\} \right): k=0, 1, \dots, N-1\right\},
 $$ which are generated by a vector
$
\left(\frac1N, \frac{h_2}N, \dots, \frac{h_d}N\right)
$
with  $\gcd(h_2, N) = \dots = \gcd(h_d, N) = 1$, are common examples of Latin hypercube sets. 

Just as in the case of permutation set in dimension $d=2$,  the  same questions naturally arise in dimension $d\ge 3$: {\emph{For which potentials $f$ do  Latin hypercubes  minimize the tensor product energy? Can one identify specific lattices or other Latin hypercube sets which minimize large classes of energies?}}  

\subsection{Main results}
Even though complete answers to the aforementioned questions appear to be out of reach for now, we provide a plethora of interesting results 
which open up a new direction in energy minimization. In particular, we prove the following:

\begin{itemize}
\item As a first indication that permutation  and Latin hypercube sets might naturally be minimizers of tensor product energies in all dimensions $d\ge 2$, we show that points in (local) minimizers cannot coincide in any coordinate provided that the potential $f$ generates sufficient repulsion at small scales (Theorem \ref{prop:coordinate_improvement}). 
\item  In Theorem \ref{thm:optimality}, which can be considered the main result of the paper, we prove that  in all dimensions $d\ge 2$ {\emph{Latin hypercube sets with a single distance vector}} (in other words, sets of maximal separation in which each interaction contributes the same amount to the energy) minimize the tensor product energy with potential $f$  whenever the energy with potential $g= \log f $ on $\mathbb T$ is minimized by equispaced points. Sufficient conditions for this property include the following:  

(a) $f$ is logarithmically convex and decreasing on $(0,1/2]$, or 

(b) $\log f(t)$ 
is absolutely monotone with respect to $s = \cos (2\pi t)$.

(See \S\ref{sec:equispaced} for more details) 
\item We  exhibit many  sets in dimensions with a single distance vector in dimensions $d\ge 2$. In fact, in Theorem \ref{thm:unique_dist_lat} we characterize all the lattices with a single distance vector. We also give many examples of non-lattice sets with this property. This provides a large collection of point configurations which satisfy this tensor  product version of universal optimality. 
\item In particular, the $3$- and $5$-point Fibonacci lattices are easily seen satisfy the assumptions of Theorem \ref{thm:optimality}.  Thus the $5$-point Fibonacci  lattice $\Phi_5$ minimizes the energy for  potentials $f$ as above, see Corollary \ref{cor:5}, i.e. $\Phi_5$ is universally optimal with respect to these  classes of potentials (and a similar result holds for the $3$-point Fibonacci lattice, Corollary \ref{cor:3}). 
\item Moreover, the $5$-point Fibonacci  lattice $\Phi_5$  minimizes the tensor product energy \eqref{e.energy} among all $5$-point permutation sets if $f$ is only assumed to be decreasing  on $(0,1/2]$, Lemma \ref{lem:5perm}. 
\item If $f$ is decreasing and logarithmically convex on $(0,1/2]$, lattices have the smallest energy among all sets of the same {\emph{permutation type}} (see Definition \ref{def.perm_type}), i.e. energy minimization under these restriction  yields permutation sets (Theorem \ref{thm:permtype}). As a corollary, under the same conditions every lattice is a local energy  minimizer (Corollary \ref{cor:lattice}).
\item We also address energy of $4$- and $6$-point sets on  $\mathbb{T}^2$. For $N=4$,  we find  global minimizers of the energy  when $f$ is differentiable on $(0, 1)$, and decreasing and logarithmically convex on $(0,1/2]$ (Proposition \ref{prop:n4}), which extends a numerical result from \cite{HO16} to a wide range of potentials. 

Moreover, we observe that a $4$-point permutation set almost never minimizes a tensor product energy and is never the unique minimizer (Remark \ref{rem:n4not}.

\item For $N=6$ a similar construction yields a minimizer among all $6$-point  configurations restricted to a natural choice of permutation type (Proposition \ref{prop:n6}). 
\end{itemize}

In the Appendix we provide a sampler  of  relevant potentials and discuss their properties.

\section{Optimality of Latin hypercube sets with a single   distance vector}\label{s.1dist}

In this section  we will present the aforementioned   collection of point configurations which minimize broad classes of tensor product energies (i.e. possess a certain version of universal optimality). To understand the relevant properties of the underlying potentials, we first discuss energy on the circle.

\subsection{Equispaced points as energy minimizers on the circle}\label{sec:equispaced}


While the one-dimensional case $\bbT$ does not exhibit tensor product structure, it is an important starting point for our investigations. In this setting, if all points repel each other equally, it is natural to expect the energy  $$ E_f(x^0, x^1, \dots, x^{N-1}) =  \sum_{\substack{k, \ell=0 \\ k \neq \ell}}^{N-1}f ({{x}}^k -  {{x}}^\ell)$$ to be minimized by \emph{equispaced points}
$$
\{ {{x}}^0, {{x}}^1,\dots, {{x}}^{N-1} \} = \left\{\frac0N, \frac1N, \dots, \frac{N-1}N\right\} \subset \bbT.
$$
In the context of numerical integration, the optimality of equispaced points was proven in a broad sense in \cite{Zen77}. However, in the general energy setting, the precise characterizations of  the potentials $f$ for which this happens is still elusive. Some known sufficient conditions on $f$  are formulated below. 
Recall that we always assume $f$  (as a $1$-periodic function on $\mathbb R$) to be even, i.e. $f(t) = f(1-t)$.

\begin{prop} \label{prop:equispaced} 
Let $N \in \bbN$. Then
\begin{equation}\label{eq:equispace}
E_f(x^0, x^1, \dots, x^{N-1}) \geq E_f\left(\frac0N, \frac1N, \dots, \frac{N-1}N\right)
\end{equation}
for all $x^0, x^1, \dots, x^{N-1} \in [0, 1)$   
provided that $f$ satisfies  any one of the following conditions:
\begin{itemize}
    \item [(i)] $f$ is convex and decreasing on $(0, 1/2]$;

    \item [(ii)] $f(t) = h(\cos(2\pi t))$ where $h: [-1, 1) \rightarrow [0, \infty)$  
     is absolutely monotone (up to an additive constant), i.e. $h$ is  a $C^\infty$-function satisfying  $h^{(k)}(s) \geq 0$ for all $s \in (-1, 1)$ and all $k \in \bbN$; 

    \item [(iii)]  $f\in C (\mathbb T)$ and $f(t) = \sum_{m=0}^\infty a_m \cos(2\pi mt)$ pointwise with  $a_m \geq 0$ for all $m > 0$, and $a_m = 0$ for all $m > 0$ with $N | m$.
\end{itemize}
\end{prop}

Condition (i) is classical and goes back to \cite{Fej56}, see also \cite{BHS12}. Condition (ii) is a special case $d=1$ of the celebrated theorem of Cohn and Kumar \cite{CK07} on universal optimality on the sphere $\mathbb S^d$.  It is not hard to see that these conditions overlap, but do not imply each other. Observe that these two conditions apply to all values of $N \in \mathbb N$ simultaneously. In fact, for a given value of $N$,  the positivity of $h^{(k)}$ in condition (ii) only needs to be checked up to order $k \leq N-1$. Condition (iii) is a simple calculus exercise (positivity of the coefficients even implies that the Fourier series converges absolutely and uniformly). \textcolor{black}{Fejér kernels
$$
f^{}_\nu(t) = \frac1\nu \left(\frac{\sin(\nu\pi t)}{\sin(\pi t)}\right)^2 = 1 + 2 \sum_{m = 1}^{\nu-1} \left(1-\frac m\nu\right) \cos(2\pi mt)
$$
are examples of potentials for which (iii) holds, but neither (i) nor (ii) applies for a fixed $N \geq \nu \geq3$.} 

None of the conditions (i)--(iii) are  necessary.  Positive definiteness of $f$ (equivalent to positivity of the coefficients in the cosine Fourier series) is a simple necessary condition for \eqref{eq:equispace} to hold  for all $N$ large enough,  
but it is not sufficient.


If  \eqref{eq:equispace} holds, i.e.  $N$ equispaced points in $\bbT$ minimize the energy $E_f$, the potential $f$ will be called \emph{$N$-equispacing}. As an example, the aforementioned potential \eqref{eq:qmc_potential} arising in quasi-Monte Carlo integration is $N$-equispacing for all $N \in \bbN$ and for all $p > 0$ as it fulfills both (i) and (ii) (as follows from \cite[equation (13)]{Leh85}) above.



\subsection{Sets with a single distance vector}

Since tensor product energies depend on distances in individual coordinates independently, rather than on actual distances between points, it is natural to consider a vector-valued analogue of distances.   For points $\bx, \by \in \bbT^d$ define the \emph{distance vector} between them as 
$$
\dv(\bx, \by) = (\delta(x_{\tau(1)}, y_{\tau(1)}), \dots, \delta(x_{\tau(d)}, y_{\tau(d)})),
$$
where $\delta$ denotes the geodesic distance on $\mathbb T$ as in \eqref{eq:geod_dist} and $\tau$ is a permutation of $\{1, \dots, d\}$ such that
$$
\delta(x_{\tau(1)}, y_{\tau(1)}) \leq \dots \leq \delta(x_{\tau(d)}, y_{\tau(d)}),
$$
i.e. we take the geodesic distance in every component and sort them increasingly. For $X = \{\bx^k\}_{k=0}^{N-1} \sbse \bbT^d$ let
$$
\dv(X) = \{\dv(\bx^k, \bx^\ell): k \neq \ell\}
$$
be the set of all distance vectors among points in $X$. We shall be interested in  Latin hypercube sets $X \sbse \bbT^d$ satisfying
$$
\#\dv(X) = 1,
$$
i.e. having only {\emph{a single distance vector}}.  We remark that such sets are natural tensor product analogues  of one-distance sets which have been well studied in  other geometric contexts, e.g. on spheres (simplices with center of mass in $\mathbf 0$) and projective spaces (equiangular tight frames), see \cite{CKM16}. However, this tensor product setting does not seem to have previously  appeared in the literature.

Since for the tensor product energy $E_F$ the interaction between two points only depends on the distance vector between them, we see that for a set $X \sbse \bbT^d, \#X = N$ with a single distance vector $\{(\rho_1, \dots, \rho_d)\} = \dv(X)$ we have
\begin{align} \label{eq:uds_energy}
    E_F(X) = (N^2-N) \prod_{i=1}^d f(\rho_i).
\end{align}
The following lemma counts the frequency of different entries in  the common distance vector in  such a set.  The main point of this computation is to show that the frequency distribution coincides with the distribution of distances in $d$ copies of $N$ equispaced points  in   $\mathbb T$.

\begin{lem} \label{lem:double_counting}
    Let $X \sbse \bbT^d, \#X=N$ be a Latin hypercube set with a  single distance vector, whose common distance vector is given by $\{(\rho_1, \dots, \rho_d)\} = \dv(X)$. Then for every $r = 1, \dots, \lfloor N/2 \rfloor$ it holds
    $$
    \#\left\{i=1, \dots, d: \rho_i = \frac rN\right\} =  \begin{cases}
        \frac{2d}{N-1}, & r < \frac N2, \\
        \frac d{N-1}, &  r = \frac N2.
    \end{cases}
    $$
\end{lem}

\begin{proof}
    We will show, via double counting, that
    \begin{align*}
        & (N^2-N) \cdot \#\left\{i=1, \dots, d: \rho_i = \frac rN\right\} \\
        = & d \cdot \#\left\{(k, \ell) \in \{0, 1, \dots, N-1\}^2: \delta\left(\frac kN, \frac \ell N\right) = \frac rN\right\},
    \end{align*}
    from which the claim follows since
    $$
    \#\left\{(k, \ell) \in \{0, 1, \dots, N-1\}^2: \delta\left(\frac kN, \frac \ell N\right) = \frac rN\right\} = \begin{cases}
        2N, &  r < \frac N2, \\
        N ,&  r = \frac N2.
    \end{cases}
    $$
    Since $X$ is a Latin hypercube, there are permutations $\sigma_1, \dots, \sigma_d$ 
    such that 
    $$
    X = \{\bx^k\}_{k=0}^{N-1} = \left\{\left(\frac{\sigma_1(k)}N, \dots, \frac{\sigma_d(k)}N\right): k=0, 1, \dots, N-1\right\}
    $$
    (with $\sigma_1 = \id$ the identity on $\{0, 1, \dots, N-1\}$). Consider the array $A = [a_{(k, \ell), i}]$ with rows indexed by $(k, \ell) \in \{0, 1, \dots, N-1\}^2, k \neq \ell$ and columns indexed by $i=1, \dots, d$, with entries
    $$
    a_{(k, l), i} = \delta\left(\frac{\sigma_i(k)}N, \frac{\sigma_i(\ell)}N\right).
    $$
    We will count for how many entries it holds $a_{(k, \ell), i} = \frac rN$. The row corresponding to $(k, \ell)$ consists, by definition, of the entries of $\dv(\bx^k, \bx^\ell)$ in some order. By the single distance property, this row contains exactly
    $$
    \#\left\{i=1, \dots, d: \rho_i = \frac rN\right\}
    $$
    such entries, so that there are
    $$
    (N^2-N) \cdot \#\left\{i=1, \dots, d: \rho_i = \frac rN\right\}
    $$
    in total of them in $A$. On the other hand, since $\sigma_i$ is a permutation, in the $i$-th column we have
    \begin{align*}
        & \#\left\{(k, \ell) \in \{0, 1, \dots, N-1\}^2: \delta\left(\frac {\sigma_i(k)}N, \frac {\sigma_i(\ell)}N\right) = \frac rN\right\} \\
        = & \#\left\{(k, \ell) \in \{0, 1, \dots, N-1\}^2: \delta\left(\frac {k}N, \frac {\ell}N\right) = \frac rN\right\},
    \end{align*}
    so in all of $A$ there are
    $$
    d \cdot \#\left\{(k, \ell) \in \{0, 1, \dots, N-1\}^2: \delta\left(\frac {k}N, \frac {\ell}N\right) = \frac rN\right\}
    $$
    such entries. This finishes the proof.
\end{proof}

For an $N$-point Latin hypercube set with a single distance vector to exist in $\bbT^d$ one thus needs to fulfill the divisibility condition
\begin{align} \label{eq:div}
    \frac{N-1}{\gcd(2, N-1)} \bigg| d,
\end{align}
in particular for a given dimension $d$ they can only exist up to $N \leq 2d+1$ points.

\subsection{Latin hypercube sets with a single distance vector as energy minimizers}

The following is the main energy minimization result of this paper. It states  that as long as $\log f$ is $N$-equispacing, then an  $N$-point Latin hypercube set with a single distance vector  $X\subset \mathbb T^d$ minimizes  the corresponding tensor product energy $E_F$.

\begin{thm} \label{thm:optimality}
    Let $X \sbse \bbT^d, \#X = N$ be a Latin hypercube set with a single distance vector. Let $f: \bbT \rightarrow (0, \infty)$ be a potential, such that  the potential $g = \log f: \bbT \rightarrow \bbR$ is $N$-equispacing. Then 
    $$
    E_F(Y) \geq E_F(X)
    $$
    for all $Y \sbse \bbT^d, \#Y = N$.
\end{thm}

According to Proposition \ref{prop:equispaced}, the result of Theorem \ref{thm:optimality} can be applied  if $\log f$ is decreasing and  convex (i.e. $f$ is logarithmically convex)  as a function of geodesic distance or if $\log f (t)$ is absolutely monotone when  rewritten as a function of $s=\cos (2\pi t) $, see e.g. \cite{GQ10} for more information on logarithmic absolute monotonicity. 

Thus, as mentioned earlier, this result can be viewed as an analogue of the classical universal optimality results on spheres and projective spaces \cite{CK07, CKM16} in the sense that one configuration minimizes  the energy for a large class of potentials. In fact, in \cite{CK07} universal optimality was  shown  using the linear programming method for {\emph{sharp configurations}}, i.e. sets which are spherical designs of high order and have few pairwise distances. The sets with a single distance vector have an extreme version of the latter property in the tensor product sense. 


\begin{proof}[Proof of Theorem \ref{thm:optimality}]
    Let $Y = \{\by^0, \by^1, \dots, \by^{N-1}\} \sbse \bbT^d, \#Y = N$. Since the exponential function is convex, we can apply Jensen's inequality to obtain
    \begin{align*}
        E_F(Y) & = \sum_{\substack{k, \ell = 0 \\ k \neq \ell}}^{N-1} \prod_{i=1}^d f(y_i^k-y_i^\ell) = \sum_{\substack{k, \ell = 0 \\ k \neq \ell}}^{N-1} \exp\left(\sum_{i=1}^d g(y_i^k-y_i^\ell)\right) \\
        & \geq (N^2-N) \exp\left(\frac1{N^2-N} \sum_{\substack{k, \ell = 0 \\ k \neq \ell}}^{N-1} \sum_{i=1}^d g(y_i^k-y_i^\ell) \right).
    \end{align*}
    By assumption, for all $i=1, \dots, d$ we have
    \begin{align*}
        \sum_{\substack{k, \ell = 0 \\ k \neq \ell}}^{N-1} g(y_i^k-y_i^\ell) & \geq \sum_{\substack{k, \ell = 0 \\ k \neq \ell}}^{N-1} g\left(\frac kN - \frac \ell N\right) \\
    & = \begin{cases}
        2N \sum_{r=1}^{(N-1)/2} g\left(\frac rN\right), &  N \text{ odd}, \\
        2N \sum_{r=1}^{(N-2)/2} g\left(\frac rN\right) + Ng\left(\frac N2\right), &  N \text{ even}
    \end{cases}
    \end{align*}
    and by Lemma \ref{lem:double_counting} thus
    \begin{align*}
        \frac1{N^2-N} \sum_{\substack{k, \ell = 0 \\ k \neq \ell}}^{N-1} \sum_{i=1}^d g(y_i^k-y_i^\ell) & \geq \frac d{N-1} \begin{cases}
        2\sum_{k=1}^{(N-1)/2} g\left(\frac kN\right) ,&  N \text{ odd}, \\
        2\sum_{k=1}^{(N-2)/2} g\left(\frac kN\right) + g\left(\frac N2\right), &  N \text{ even}
    \end{cases} \\
    & = \sum_{r=1}^{\lfloor N/2 \rfloor} \#\left\{i=1, \dots, d: \rho_i = \frac rN\right\} \cdot g\left(\frac rN\right),
    \end{align*}
    where $\{(\rho_1, \dots, \rho_d)\} = \dv(X)$ is the unique distance vector in $X$. In total we get
    $$
    E_F(Y) \geq (N^2-N) \exp\left(\sum_{i=1}^d g(\rho_i)\right) = E_F(X)
    $$
    by \eqref{eq:uds_energy}, finishing the proof.
\end{proof}

We may readily apply this result to the potential \eqref{eq:qmc_potential} arising from QMC integration. As was already observed in \cite{HO16}, the potential $f_p$ is logarithmically convex as long as $p \leq 6$. Combining condition (i) of  Proposition \ref{prop:equispaced}  and Theorem \ref{thm:optimality} we obtain the following.

\begin{cor} \label{cor:opt_qmc}
    Let $X \sbse \bbT^d, \#X = N$ be a Latin hypercube set with a single distance vector. Then $X$ is a minimizer for the energy given by the potential
    $$
    f_p(t) = 1 + \frac p2 \left(\frac16 - t + t^2\right)
    $$
    for all $0 < p \leq 6$.
\end{cor}

 \begin{figure}[htp]
						\begin{tikzpicture}
							\begin{axis}[
								plot box ratio = 1 1,
								xmin=-1/21,xmax=22/21,ymin=-1/21,ymax=22/21,
								font=\footnotesize,
								height=0.3\textwidth,
								width=0.3\textwidth,
								xticklabel = \empty,
								yticklabel = \empty,
								ytick style={draw=none},
								xtick style={draw=none},
								]
								\addplot[only marks,black,mark=*,mark size=2pt,mark options={solid}] coordinates {
									(0, 0) (1/5, 3/5) (2/5, 1/5) (3/5, 4/5) (4/5, 2/5)
								};
								\addplot[gray,mark options={solid}] coordinates {
									(0, -0.05) (0, 1.05)
								};
								\addplot[gray,mark options={solid}] coordinates {
									(1.0, -0.05) (1.0, 1.05)
								};
								\addplot[gray,mark options={solid}] coordinates {
									(-0.05, 0) (1.05, 0)
								};
								\addplot[gray,mark options={solid}] coordinates {
									(-0.05, 1.0) (1.05, 1.0)
								};
								\addplot[gray,mark options={solid}, dotted] coordinates {
									(0, 1/5) (1, 1/5)
								};
								\addplot[gray,mark options={solid}, dotted] coordinates {
									(0, 2/5) (1, 2/5)
								};
								\addplot[gray,mark options={solid}, dotted] coordinates {
									(1/5, 0) (1/5, 1)
								};
								\addplot[gray,mark options={solid}, dotted] coordinates {
									(2/5, 0) (2/5, 1)
								};
								\addplot[gray,mark options={solid}, dotted] coordinates {
									(0, 3/5) (1, 3/5)
								};
								\addplot[gray,mark options={solid}, dotted] coordinates {
									(0, 4/5) (1, 4/5)
								};
								\addplot[gray,mark options={solid}, dotted] coordinates {
									(3/5, 0) (3/5, 1)
								};
								\addplot[gray,mark options={solid}, dotted] coordinates {
									(4/5, 0) (4/5, 1)
								};
							\end{axis}
						\end{tikzpicture}
						\begin{tikzpicture}
							\begin{axis}[
								plot box ratio = 1 1,
								xmin=-1/21,xmax=22/21,ymin=-1/21,ymax=22/21,
								font=\footnotesize,
								height=0.3\textwidth,
								width=0.3\textwidth,
								xticklabel = \empty,
								yticklabel = \empty,
								ytick style={draw=none},
								xtick style={draw=none},
								]
								\addplot[only marks,black,mark=*,mark size=2pt,mark options={solid}] coordinates {
									(0, 0) (1/8, 5/8) (2/8, 2/8) (3/8, 7/8) (4/8, 4/8) (5/8, 1/8) (6/8, 6/8) (7/8, 3/8)
								};
								\addplot[gray,mark options={solid}] coordinates {
									(0, -0.05) (0, 1.05)
								};
								\addplot[gray,mark options={solid}] coordinates {
									(1.0, -0.05) (1.0, 1.05)
								};
								\addplot[gray,mark options={solid}] coordinates {
									(-0.05, 0) (1.05, 0)
								};
								\addplot[gray,mark options={solid}] coordinates {
									(-0.05, 1.0) (1.05, 1.0)
								};
								\addplot[gray,mark options={solid}, dotted] coordinates {
									(0, 1/8) (1, 1/8)
								};
								\addplot[gray,mark options={solid}, dotted] coordinates {
									(0, 2/8) (1, 2/8)
								};
								\addplot[gray,mark options={solid}, dotted] coordinates {
									(1/8, 0) (1/8, 1)
								};
								\addplot[gray,mark options={solid}, dotted] coordinates {
									(2/8, 0) (2/8, 1)
								};
								\addplot[gray,mark options={solid}, dotted] coordinates {
									(0, 3/8) (1, 3/8)
								};
								\addplot[gray,mark options={solid}, dotted] coordinates {
									(0, 4/8) (1, 4/8)
								};
								\addplot[gray,mark options={solid}, dotted] coordinates {
									(3/8, 0) (3/8, 1)
								};
								\addplot[gray,mark options={solid}, dotted] coordinates {
									(4/8, 0) (4/8, 1)
								};
								\addplot[gray,mark options={solid}, dotted] coordinates {
									(0, 5/8) (1, 5/8)
								};
								\addplot[gray,mark options={solid}, dotted] coordinates {
									(0, 6/8) (1, 6/8)
								};
								\addplot[gray,mark options={solid}, dotted] coordinates {
									(5/8, 0) (5/8, 1)
								};
								\addplot[gray,mark options={solid}, dotted] coordinates {
									(6/8, 0) (6/8, 1)
								};
								\addplot[gray,mark options={solid}, dotted] coordinates {
									(0, 7/8) (1, 7/8)
								};
								\addplot[gray,mark options={solid}, dotted] coordinates {
									(7/8, 0) (7/8, 1)
								};
							\end{axis}
						\end{tikzpicture}
						\begin{tikzpicture}
							\begin{axis}[
								plot box ratio = 1 1,
								xmin=-1/21,xmax=22/21,ymin=-1/21,ymax=22/21,
								font=\footnotesize,
								height=0.3\textwidth,
								width=0.3\textwidth,
								xticklabel = \empty,
								yticklabel = \empty,
								ytick style={draw=none},
								xtick style={draw=none},
								]
								\addplot[only marks,black,mark=*,mark size=2pt,mark options={solid}] coordinates {
									(0/13, 0/13) (1/13, 8/13) (2/13, 3/13) (3/13, 11/13) (4/13, 6/13) (5/13, 1/13) (6/13, 9/13) (7/13, 4/13) (8/13, 12/13) (9/13, 7/13) (10/13, 2/13) (11/13, 10/13) (12/13, 5/13)
								};
								\addplot[gray,mark options={solid}] coordinates {
									(0, -0.05) (0, 1.05)
								};
								\addplot[gray,mark options={solid}] coordinates {
									(1.0, -0.05) (1.0, 1.05)
								};
								\addplot[gray,mark options={solid}] coordinates {
									(-0.05, 0) (1.05, 0)
								};
								\addplot[gray,mark options={solid}] coordinates {
									(-0.05, 1.0) (1.05, 1.0)
								};
								\addplot[gray,mark options={solid}, dotted] coordinates {
									(0, 1/13) (1, 1/13)
								};
								\addplot[gray,mark options={solid}, dotted] coordinates {
									(0, 2/13) (1, 2/13)
								};
								\addplot[gray,mark options={solid}, dotted] coordinates {
									(1/13, 0) (1/13, 1)
								};
								\addplot[gray,mark options={solid}, dotted] coordinates {
									(2/13, 0) (2/13, 1)
								};
								\addplot[gray,mark options={solid}, dotted] coordinates {
									(0, 3/13) (1, 3/13)
								};
								\addplot[gray,mark options={solid}, dotted] coordinates {
									(0, 4/13) (1, 4/13)
								};
								\addplot[gray,mark options={solid}, dotted] coordinates {
									(3/13, 0) (3/13, 1)
								};
								\addplot[gray,mark options={solid}, dotted] coordinates {
									(4/13, 0) (4/13, 1)
								};
								\addplot[gray,mark options={solid}, dotted] coordinates {
									(0, 5/13) (1, 5/13)
								};
								\addplot[gray,mark options={solid}, dotted] coordinates {
									(0, 6/13) (1, 6/13)
								};
								\addplot[gray,mark options={solid}, dotted] coordinates {
									(5/13, 0) (5/13, 1)
								};
								\addplot[gray,mark options={solid}, dotted] coordinates {
									(6/13, 0) (6/13, 1)
								};
								\addplot[gray,mark options={solid}, dotted] coordinates {
									(0, 7/13) (1, 7/13)
								};
								\addplot[gray,mark options={solid}, dotted] coordinates {
									(7/13, 0) (7/13, 1)
								};
								\addplot[gray,mark options={solid}, dotted] coordinates {
									(0, 8/13) (1, 8/13)
								};
								\addplot[gray,mark options={solid}, dotted] coordinates {
									(0, 9/13) (1, 9/13)
								};
								\addplot[gray,mark options={solid}, dotted] coordinates {
									(8/13, 0) (8/13, 1)
								};
								\addplot[gray,mark options={solid}, dotted] coordinates {
									(9/13, 0) (9/13, 1)
								};
								\addplot[gray,mark options={solid}, dotted] coordinates {
									(0, 10/13) (1, 10/13)
								};
								\addplot[gray,mark options={solid}, dotted] coordinates {
									(0, 11/13) (1, 11/13)
								};
								\addplot[gray,mark options={solid}, dotted] coordinates {
									(10/13, 0) (10/13, 1)
								};
								\addplot[gray,mark options={solid}, dotted] coordinates {
									(11/13, 0) (11/13, 1)
								};
								\addplot[gray,mark options={solid}, dotted] coordinates {
									(0, 12/13) (1, 12/13)
								};
								\addplot[gray,mark options={solid}, dotted] coordinates {
									(12/13, 0) (12/13, 1)
								};
							\end{axis}
						\end{tikzpicture}
\caption{The $5$-, $8$- and $13$-point Fibonacci lattices.}
\label{fig:fib_lats}
\end{figure}
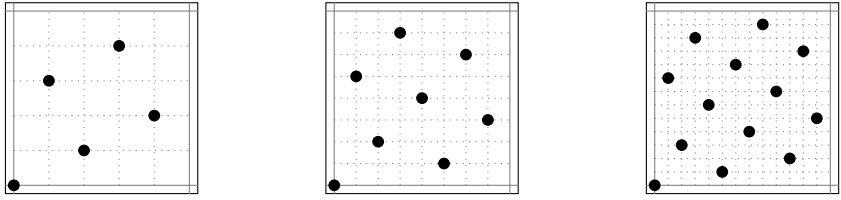

Of course, Theorem \ref{thm:optimality} is only valuable if we can demonstrate that sufficiently many sets of this form exist. We dedicate the next three subsections to constructing a large number of such sets, starting with the case $d=2$ with particular emphasis on $N=5$, the largest $N$ for which there is such a set in $\mathbb{T}^2$ (see \eqref{eq:div}).  

\subsection{Optimality of the $3$- and $5$-point Fibonacci lattices}

We observe that the 
$5$-point Fibonacci lattice
$$
\Phi_5 = \left\{\left(\frac05, \frac05\right), \left(\frac15, \frac35\right), \left(\frac25, \frac15\right), \left(\frac35, \frac45\right), \left(\frac45, \frac25\right)\right\}
$$
has a single distance vector $\left(\frac15, \frac25 \right)$, see Figure \ref{fig:fib_lats}, in particular $
E_F(\Phi_5) = 20 f\left(\frac15\right) f\left(\frac25\right)
$
as in \eqref{eq:uds_energy}. Therefore, Theorem \ref{thm:optimality} applies in this case and we immediately obtain optimality of this configuration for a variety of potentials.


\begin{cor}\label{cor:5}
    If  $\log f$ is $5$-equispacing, then the  $5$-point Fibonacci lattice $\Phi_5$ is an energy minimizer for the corresponding energy $E_F$. In particular, this holds if  (see Proposition \ref{prop:equispaced})
    \begin{itemize}
    \item $f$ is decreasing and logarithmically convex on $(0,1/2]$; or
    \item $f (t) = h (\cos (2\pi t))$ and $\log h$ is absolutely monotone on $[-1,1)$. 
    \end{itemize}
\end{cor}

We further observe that the $5$-point Fibonacci lattice has the lowest energy  among all other $5$-point permutation sets merely under the assumption that $f$ is decreasing  on $(0,1/2]$.

\begin{lem}\label{lem:5perm}
Assume that $f$ is decreasing  on $(0,1/2]$ (or just that $f(2/5)< f(1/5)$). 
    Then the  $5$-point Fibonacci lattice $\Phi_5$ minimizes the energy $E_F$  among all $5$-point permutation sets.
\end{lem}

\begin{proof}
    We note that permutation sets on $\mathbb{T}^2$ project down to equispaced points on $\mathbb{T}^1$ in both the $x$- and $y$-coordinates. The one-dimensional energy of either of these projections is $5f(1/5)+5f(2/5)$, so the energy of any permutation set in $\mathbb{T}^2$ must be some sum of the form $\sum_{i=1}^{10}a_ib_i$ where five of the $a_i$ are $f(1/5)$, five of the $a_i$ are $f(2/5)$, five of the $b_i$ are $f(1/5)$ and five of the $b_i$ are $f(2/5)$. By the rearrangement inequality \cite{HLP88} we note then that for any permutation set $X$, $$E_F(X)\ge \sum_{i=1}^{10}f(1/5)f(2/5)=10f(1/5)f(2/5),$$ with equality only holding when $f(1/5)=f(2/5)$. Since $10f(1/5)f(2/5)$ is the energy of the $5$-point Fibonacci lattice, we have the desired result. 
\end{proof}

We also notice that the $3$-point Fibonacci lattice  is   equivalent, up to symmetries,  to the diagonal lattice, consisting of points  $\left(0, 0 \right)$, $\left(\frac13, \frac13 \right)$, and   $\left(\frac23, \frac23 \right)$, and has a single distance vector $\left(\frac13, \frac13 \right)$. Thus Theorem \ref{thm:optimality} also applies in this case and yields an identical corollary:

\begin{cor}\label{cor:3}
    If  $\log f$ is $3$-equispacing, then the  $3$-point Fibonacci lattice $\Phi_4$ minimizes the  tensor product  energy $E_F$.
    \end{cor}
   \noindent  A  statement similar to this has already been observed in \cite{Nag25b} and a different proof of this fact in the case when $f$ is logarithmically convex is given in \S\ref{s.type}.

Finally we would like to mention some other previously known results  about the Fibonacci lattice  and compare them to our results above. All of the prior results concerned specific potentials arising from QMC integration. 

\begin{itemize}
\item It was proved in \cite{Nag25b} that for the potential $f_p$ in \eqref{eq:qmc_potential} with $0 < p \leq 9$ the $5$-point Fibonacci lattice minimizes $E_F(X)$ among all $5$-point configurations. For this particular potential, Corollary \ref{cor:opt_qmc} yields a smaller range of $p \leq 6$. 
We do however want to point out that the method to prove Theorem \ref{thm:optimality} is completely elementary and does not rely on any computer assisted calculations.

\item Let again $f_p$ be the potential as in \eqref{eq:qmc_potential}. Via an optimization algorithm it was shown in \cite{HO16} that the $N$-point Fibonacci lattice for $N \in \{8, 13\}$ (which has multiple distance vectors, see Figure \ref{fig:fib_lats}) is the energy minimizer for $p \in \{1, 6\}$ among all $N$-point configurations.

    \end{itemize}


\subsection{Classification of lattices with a single distance vector}

In this section we will classify all lattices with a single distance vector. For small values of $N$ this can be checked straightforwardly. Note that for each $N=2, 3, 4$  there is, up to torus symmetry,  
 only one lattice in $\bbT^d$ with $N$ points:   the corresponding generating vector is 
$$
\left(\frac1N, \dots, \frac1N\right).
$$
The lattices with $N=2, 3$ points are easily seen to 
have a single distance vector, while the case of $N=4$ defines precisely two distance vectors. 

\begin{thm} \label{thm:unique_dist_lat}
    A lattice $\Lambda \sbse [0, 1)^d, \#\Lambda = N \geq 3$ has a single distance vector 
    if and only if $N$ is prime and the generating vector of $\Lambda$ is of the form (up to torus symmetry)
    $$
    \bigg(\underbrace{\frac1N, \dots, \frac1N}_{m \text{ times}}, \underbrace{\frac2N, \dots, \frac2N}_{m \text{ times}}, \dots, \underbrace{\frac nN, \dots, \frac nN}_{m \text{ times}}\bigg)
    $$
    with $N = 2n+1$ and $d = mn$.
\end{thm}

\begin{proof}
    Let $\Lambda \sbse \bbT^d$ be a 
     lattice with a single distance vector such that $\#\Lambda = N \geq 3$. By applying torus symmetries we may assume the generating vector to have the form
    $$
    \mathbf h = \left(\frac{h_1}N, \dots, \frac{h_d}N\right)
    $$
    with $1 = h_1 \leq \dots \leq h_d \leq \lfloor N/2 \rfloor$. In particular, the distance vector among points in $\Lambda$ must be of the form
    $$
    \dv(\mathbf0, \mathbf h) = \left(\frac{h_1}N, \dots, \frac{h_d}N\right),
    $$
    by Lemma \ref{lem:double_counting} thus
    $$
    \left(\frac{h_1}N, \dots, \frac{h_d}N\right) = \bigg(\underbrace{\frac1N, \dots, \frac1N}_{m \text{ times}}, \underbrace{\frac2N, \dots, \frac2N}_{m \text{ times}}, \dots, \underbrace{\frac nN, \dots, \frac nN}_{m \text{ times}}\bigg)
    $$
    with $n = \frac{N-1}2 \in \bbN$ and $m = \frac{2d}{N-1}$. Furthermore, $N$ must be relatively prime to $1, \dots, \lfloor N/2\rfloor$, so that it must already be prime itself.

    It remains to show that any lattice with generating vector
    $$
    \bigg(\underbrace{\frac1N, \dots, \frac1N}_{m \text{ times}}, \underbrace{\frac2N, \dots, \frac2N}_{m \text{ times}}, \dots, \underbrace{\frac nN, \dots, \frac nN}_{m \text{ times}}\bigg)
    $$
    with $d=mn$ and $N=2n+1$ prime is a lattice with a single distance vector. Indeed, for an arbitrary $k \in \{1, \dots, N-1\}$ consider the distance vector from $\mathbf{0} = (0, \dots, 0)$ to
    $$
    \bx^k = \left(\left\{\frac{h_1 k}N\right\}, \dots, \left\{\frac{h_d k}N\right\}\right)
    $$
    with $h_1 = \dots = h_m = 1, h_{m+1} = \dots = h_{2m} = 2, \dots, h_{(n-1)m+1} = \dots = h_{nm} = n$. We will show that this distance vector does not depend on $k$. If there were $h, \tilde h \in \{1, \dots, n\}, h \neq \tilde h$ with
    $$
    \left\{\frac{hk}N\right\} = \left\{\frac{\tilde hk}N\right\},
    $$
    then we would either get
    $$
    hk \equiv \tilde h k \mod N \quad \Rightarrow \quad h = \tilde h
    $$
    or
    $$
    hk \equiv N - \tilde h k \mod N \quad \Rightarrow \quad (h + \tilde h) k \equiv 0 \mod N \quad \Rightarrow \quad h + \tilde h = N,
    $$
    both leading to contradictions since $h \neq \tilde h$ and $h + \tilde h \leq 2n < N$. Thus
    $$
    \dv(\mathbf 0, \bx^k) = \bigg(\underbrace{\frac1N, \dots, \frac1N}_{m \text{ times}}, \underbrace{\frac2N, \dots, \frac2N}_{m \text{ times}}, \dots, \underbrace{\frac nN, \dots, \frac nN}_{m \text{ times}}\bigg)
    $$
    and by the lattice structure this is also the distance vector between any other two point from the lattice.
\end{proof}

We see that if $2d+1$ is prime, the bound on $N$ mentioned after \eqref{eq:div} can be attained by such a lattice construction.


\subsection{Some non-lattice  sets with a single distance vector}

In addition to the lattices from the previous section, we can also construct  the following $4$-, $6$-, $8$-, and $9$-point Latin hypercube sets with a single distance vector in $\bbT^3$, $\bbT^{10}$, $\bbT^{14}$, and $\bbT^{12}$ respectively:
\begin{align*}
    X_4 = \left\{\frac14 (0, 0, 0), \frac14 (1, 1, 2), \frac14 (2, 3, 1), \frac14 (3, 2, 3)\right\},
\end{align*}

\begin{align*}
    X_6 = \bigg\{ & \frac16 (0, 0, 0, 0, 0, 0, 0, 0, 0, 0), \frac16 (1, 1, 1, 1, 2, 2, 2, 2, 3, 3), \\
    & \frac16 (2, 2, 2, 3, 3, 4, 5, 5, 1, 1), \frac16 (3, 3, 4, 2, 5, 1, 1, 4, 4, 5), \\
    & \frac16 (4, 5, 3, 4, 1, 3, 4, 1, 5, 4), \frac16 (5, 4, 5, 5, 4, 5, 3, 3, 2, 2) \bigg\},
\end{align*}

\begin{align*}
    X_8 = \bigg\{ & \frac18 (0, 0, 0, 0, 0, 0, 0, 0, 0, 0, 0, 0, 0, 0), \frac18 (1, 1, 1, 1, 2, 2, 2, 2, 3, 3, 3, 3, 4, 4), \\
& \frac18 (2, 2, 3, 5, 1, 4, 4, 5, 1, 6, 6, 7, 1, 3), \frac18 (3, 3, 2, 4, 3, 6, 6, 7, 6, 1, 1, 4, 5, 7), \\
& \frac18 (4, 5, 6, 3, 5, 1, 5, 1, 2, 4, 7, 1, 2, 6), \frac18 (5, 4, 7, 2, 7, 3, 7, 3, 5, 7, 4, 6, 6, 2), \\
& \frac18 (6, 7, 5, 6, 4, 5, 1, 4, 7, 2, 5, 2, 3, 1), \frac18 (7, 6, 4, 7, 6, 7, 3, 6, 4, 5, 2, 5, 7, 5)\bigg\},
\end{align*}

\begin{align*}
    X_9 = \bigg\{ & \frac19 (0, 0, 0, 0, 0, 0, 0, 0, 0, 0, 0, 0), \frac19 (1, 1, 1, 2, 2, 2, 3, 3, 3, 4, 4, 4), \\
& \frac19 (2, 2, 2, 4, 4, 4, 6, 6, 6, 8, 8, 8), \frac19 (3, 4, 5, 1, 6, 8, 1, 4, 7, 2, 3, 7), \\
& \frac19 (4, 3, 8, 7, 8, 6, 2, 8, 5, 5, 7, 3), \frac19 (5, 6, 4, 8, 1, 3, 4, 7, 1, 7, 2, 6), \\
& \frac19 (6, 5, 7, 5, 3, 1, 5, 2, 8, 1, 6, 2), \frac19 (7, 8, 3, 6, 5, 7, 7, 1, 4, 3, 1, 5), \\
& \frac19 (8, 7, 6, 3, 7, 5, 8, 5, 2, 6, 5, 1)\bigg\}.
\end{align*}
These point sets are again energy minimizers in the sense of Theorem \ref{thm:optimality}. It can be observed that $X_4$ and $X_8$ are digital nets \cite{DP10}.

It is easy to see that for every $N \in \bbN$ there exists a Latin hypercube set  $X \sbse \bbT^d, \#X = N$ with a single distance vector in dimension $d = (N-1)!$, namely where every coordinate of $\bbT^{(N-1)!}$ corresponds to a permutation $\sigma$ of $\{0, 1,\dots, N-1\}$ fixing $\sigma(0) = 0$. It is thus more interesting to find, for any given $N \in \bbN$, the smallest dimension $d = d(N)$ for which there exists an $N$-point Latin hypercube set with a single distance vector in $\bbT^d$. The condition \eqref{eq:div} is necessary but not sufficient.

\section{Energy minimizers have distinct coordinates}\label{s.noequal}

In the present section we  demonstrate that, under very mild assumptions on the function $f$, both global and local energy minimizers must necessarily have distinct coordinates. For example, in the two-dimensional case $\bbT^2$ this implies no two points in  the minimizing configuration can lie on the same horizontal or vertical line. This brings us one step closer to understanding when permutation sets or Latin hypercube sets minimizer tensor product energies. The result is formalized in the following theorem:

\begin{thm}\label{prop:coordinate_improvement}
	Let $f: \mathbb R  \rightarrow (0, \infty)$ be continuous, $1$-periodic, even and satisfy the following assumptions:
	\begin{itemize}		
		\item [(i)] $f(t)$ is differentiable on $(0, 1)$, 
		
		\item [(ii)] there are constants $C > 0$ and  $\varepsilon > 0$ such that  $f(t) \leq f(0) - Ct$ for all $0 \leq t < \varepsilon$. 
	\end{itemize}
	Consider an $N$-point configuration  $X = \{\bx^0, \dots, \bx^{N-1}\} \subseteq \bbT^d$ such that $x^0_1 = \dots = x^j_1$ for some  $j \geq 1$ and $x^k_1 \neq x^0_1$ for all $k > j$ (if such $k$ exist). 
	Define a perturbation of $X$ by
    $$
	X_r \coloneqq \{\bx^0(r), \bx^1, \dots, \bx^{N-1}\},
	$$
    where $\bx^0(r) = (x_1^0 + r, x_2^0, \dots, x_d^0)$ for $r \in [-1/2, 1/2)$. Then the function
    $$
    r \mapsto E_F(X_{r})
    $$
    does not have a local minimium at $r=0$.
	\end{thm}

As a direct consequence we obtain the following.

\begin{cor}\label{c.noequal}
    Let $f$ satisfy the assumptions of Theorem \ref{prop:coordinate_improvement} and let $X = \{\bx^0, \bx^1, \dots, \bx^{N-1}\} \sbse \bbT^d$ be a (local) minimizer of the corresponding energy $E_F$. Then
    $$
    x_i^k \neq x_i^\ell
    $$
    for all $i = 1, \dots, d$ and $k, \ell = 0, 1, \dots, N-1, k \neq \ell$.
\end{cor}


That is, points in minimizing configurations must necessarily have distinct coordinates.

\begin{rem*}
Assumption (ii) in Theorem  \ref{prop:coordinate_improvement} (which means that $f$ has a sharp kink at $t=0$) can be interpreted as saying that the repulsion induced by the interaction $F$ is quite strong at small distances in each coordinate. 
We observe that if $f$ is singular at zero, i.e. $f(0) = + \infty$, the statement of Corollary \ref{c.noequal}  would have been obvious. It is a known phenomenon \cite{CFP} that when the local repulsion is weak, then minimizers of energies tend to cluster. Thus, this result can be viewed as complimentary: minimizers do not cluster when the local repulsion is strong. 

Observe that assumption (ii) is satisfied by various natural kernels, in particular the aforementioned $f(t) = 1/2 - t + t^2$ for $t\in [0,1]$ which corresponds to the periodic $L_2$-discrepancy, or more generally the QMC potential \eqref{eq:qmc_potential} for $0 < p < 24$ (where it stays positive). 
\end{rem*}

To prove Theorem  \ref{prop:coordinate_improvement}, we begin with a simple   observation: if a function has a downwards-opening kink, then it  cannot have a local minimum there.

\begin{lem} \label{lem:basic_calculcus_lemma}
	Let $G: (-\varepsilon, \varepsilon) \rightarrow \mathbb{R}$ be a function with $G(r) \leq G(0) - C|r|$ in $-\varepsilon < r < \varepsilon$ for some $C > 0$. Let $H: (-\varepsilon, \varepsilon) \rightarrow \mathbb{R}$ be   differentiable at $r = 0$. Then $r \mapsto G(r) + H(r)$ does not have a local minimum at $r = 0$.
\end{lem}


\begin{proof} [Proof of Theorem \ref{prop:coordinate_improvement}]
    Without loss of generality let  $x^0_1 = \dots = x^j_1=0$ be the common coordinate. Consider the perturbation $\bx^0(r) = (r, x_2^0, \dots, x_d^0)$ while keeping all other coordinates fixed. We separate the terms of $E_F (X_r)$ that depend on $r$ and treat the rest as a constant $K_1>0$
	\begin{align*}
		E_F (X_r ) = & f(r) \cdot 2\sum_{k = 1}^{j} \prod_{i=2}^d f(x_i^0-x_i^k) 
		 +  2\sum_{k = j+1}^{N-1} f(r - x_1^k) \prod_{i=2}^d f(x_i^0-x_i^k) + K_1.
	\end{align*}
	Let
	$$
	G(r) = f(r) \cdot 2\sum_{k = 1}^{j} \prod_{i=2}^d f(x_i^0-x_i^k)
    $$
    and
    $$
	H(r)  =  2\sum_{k = j+1}^{N-1} f(r - x_1^k) \prod_{i=2}^d f(x_i^0-x_i^k) + K_1.
	$$
	By (i) and since $x_1^k \neq 0$ for $k > j$ we have that $H(r)$ is differentiable at $r=0$. Also note that
	$$
    K_2 = 2\sum_{k = 1}^{j} \prod_{i=2}^d f(x_i^0-x_i^k) > 0
	$$
	is a positive constant since $f$ only takes on positive values and $j \geq 1$ so that the sum is nonempty. Hence (ii) implies that
	$$
	G(r) \leq K_2 (f(0) - C |r|)  = G(0) - C K_2 |r|
	$$
	in a neighborhood around $r = 0$. Thus $G$ and $H$ satisfy the assumptions of Lemma \ref{lem:basic_calculcus_lemma}, hence $E_F (X_r) = G(r) +H(r)$ does not have a local minimum in $r=0$.
\end{proof}



\section{Permutation type and winding number}\label{s.type}

In this section  we show that lattices minimize tensor product energies with logarithmically convex potential $f$  among large classes of sets which share structural similarities with the lattice. In particular, lattices are local minimizers of such energies. To this end, we prove a lower bound on the energy in terms of the permutation type of the configuration.

Given a point set in  $\mathbb T^d \simeq[0,1)^d$, we can always assume that the $x_1$-coordinates are ordered, i.e. $x_1^k\le x_1^\ell$ for $k<\ell$. However the same need not hold for the other coordinates. Instead when $d\ge 2$ we can order the points by one of the other coordinates and obtain a permutation that relates this ordering to our original ordering. We obtain the following definition:
\begin{defi}\label{def.perm_type}
    For a given point set $X = \{ \bx^0, \bx^1, \dots, \bx^{N-1} \} \subset \mathbb T^d \simeq[0,1)^d$, suppose $S=(\sigma_1,\sigma_2,\dots,\sigma_d)$ is a $d$-tuple of permutations of $\{0,1,\dots,N-1\}$ satisfying
    $$
    x_i^{\sigma^{-1}_i(0)}\le x_i^{\sigma^{-1}_i(1)}\le\dots\le x_i^{\sigma^{-1}_i(N-1)}
    $$
    for all $i$. Then we say $S$ is a \emph{permutation type} of $X$.
\end{defi}

For example, we note that the Latin hypercube set $X_{\sigma_2,...,\sigma_d}$ given in Definition \ref{def:lat_hypcube} has permutation type $(\text{id},\sigma_2,...,\sigma_d)$. We will generally assume that the first coordinates are ordered, i.e. $\sigma_1$ is the identity permutation, and by translating so that $\bx^0=(0,0,...,0)$ we will also assume that $\sigma_i(0)=0$ for all $i$. The permutation type is uniquely defined if and only if all points have distinct coordinates in the sense of §\ref{s.noequal}. For any permutation $\sigma$ we will define the \emph{$k$-th winding number} associated with $\sigma$ by $$\omega_\sigma(k):=\#\{j:\sigma((j+k) \textup{ mod } N)<\sigma(j)\}.$$ 
For notational simplicity,  input values outside of the domain of $\sigma$ will be implicitly assumed to be taken modulo $N$. Conceptually, this quantity counts how many times the permutation winds around the torus, moving along it with step size $k$. With this definition we are now ready to state the lower bound for the energy.

\begin{thm}\label{thm:permtype}
    Assume the configuration  $X=\{ \bx^0, \bx^1,\dots,\bx^{N-1} \} \subset \mathbb T^d $ has unique permutation type $(\sigma_1,\sigma_2,\dots,\sigma_d)$. Then the following inequality holds for all  potentials $f$ decreasing and logarithmically convex on $(0,1/2]$:
    $$E_F(\bx^0,\dots,\bx^{N-1})\ge N\sum_{k=1}^{N-1}\prod_{i=1}^df\left(\frac{\omega_{\sigma_i}(k)}{N}\right).$$
    If $f$ is strictly decreasing, equality is obtained if and only if the point set $X = \{ \bx^0, \bx^1,\dots,\bx^{N-1} \}$ is a lattice.
\end{thm}

\begin{proof}
    By regrouping the terms,  the energy can be written in the form \begin{equation}\label{eq:perm_energy}
        E_F(\bx^0,\dots,\bx^{N-1})=\sum_{\ell=1}^{N-1}\sum_{k=0}^{N-1} F(\bx^k-\bx^{k+\ell}),
    \end{equation}
    if we identify $\bx^{k+\ell}=\bx^{k+\ell-N}$ in the case $k+\ell>N-1$. For $x,y \in [0,1) \simeq \mathbb T$ we define
    $$
    \delta_\circlearrowleft(x,y)=\begin{cases}
        y-x, & y \geq x, \\
        1+y-x, & y < x,
    \end{cases}
    $$
    which can be thought of as the \emph{counterclockwise} geodesic distance. Note that this need not always equal the geodesic distance $\delta$ defined in \eqref{eq:geod_dist}, but the periodicity and evenness of $f$ implies  that the value of  $f$ is independent of the choice between these two distances as its argument. Thus we rewrite 
    \begin{equation*}
    F\left(\bx^k-\bx^{k+\ell}\right)=\prod_{i=1}^df\left(x_i^k-x_i^{k+\ell}\right)  = \prod_{i=1}^df\left(\delta_\circlearrowleft(x_i^k,x_i^{k+\ell})\right).\end{equation*}
    Since $f$ is logarithmically convex and decreasing on $(0,1/2]$, it is log-convex on $(0,1)$, and hence  
    the function $F$ is convex as a function of $d$ variables on $(0,1)^d$. Therefore, we can apply  Jensen's inequality to the inner sum in  the right-hand side of \eqref{eq:perm_energy} and we obtain 
     \begin{align}\label{eq:jensens}
     \begin{split}
         E_F(\bx^0,\dots,\bx^{N-1}) &  =  \sum_{\ell=1}^{N-1} \sum_{k=0}^{N-1} F \left( \delta_\circlearrowleft(x_1^k,x_1^{k+\ell}), \dots, \delta_\circlearrowleft(x_d^k,x_d^{k+\ell} )\right)\\ & \ge N\sum_{\ell=1}^{N-1}F\left(\frac{1}{N}\sum_{k=0}^{N-1}\delta_\circlearrowleft(x_1^k,x_1^{k+\ell}),\dots,\frac{1}{N}\sum_{k=0}^{N-1}\delta_\circlearrowleft(x_d^k,x_d^{k+\ell})\right)\\
    & = N\sum_{\ell=1}^{N-1}\prod_{i=1}^d f\left(\frac{1}{N}\sum_{k=0}^{N-1}\delta_\circlearrowleft(x_i^k,x_i^{k+\ell})\right).
     \end{split}
    \end{align}
 Observe that $\delta_\circlearrowleft(x_i^k,x_i^{k+\ell})= 1+x_i^{k+\ell}-x_i^k$ if and only if 
 $x_i^{k+\ell}<x_i^k$, i.e.  $x_i^{\sigma^{-1}_i(\sigma_i(k+\ell))}<x_i^{\sigma^{-1}_i(\sigma_i(k))}$, which is equivalent to  $\sigma_i(k+\ell)<\sigma_i(k)$. This  occurs for $\omega_{\sigma_i}(\ell)$ values of $k$ (note that unique permutation type means that we can ignore the case of equality between coordinates). Otherwise it equals $x_i^{k+\ell}-x_i^k$. Telescoping the sum, we  conclude that 
 $$ \sum_{\ell=0}^{N-1}\delta_\circlearrowleft(x_i^k,x_i^{k+\ell}) = {\omega_{\sigma_i}(k)},$$ and thus, as desired,  $$E_F(\bx^0,\dots,\bx^{N-1})\ge N\sum_{k=1}^{N-1}\prod_{i=1}^df\left(\frac{\omega_{\sigma_i}(k)}{N}\right).$$
    
    Finally, we note that there are two cases for equality in Jensen's inequality in \eqref{eq:jensens}: either $f$ is linear on the relevant domain (which cannot be the case if $f$ is strictly decreasing and logarithmically convex on $(0,1/2]$) or the argument of $F$ is independent of $k$, i.e. \begin{equation*}
       \delta_\circlearrowleft(x_i^k-x_i^{k+\ell})=\delta_\circlearrowleft(x_i^{k'}-x_i^{k'+\ell})\text{ for all }i, \ell, k,k' ,
    \end{equation*} which by their periodic structure is fulfilled by lattices. Conversely, if this condition is satisfied we note that since a unique permutation type means none of the coordinates coincide, the $\ell=1$ case implies that our configuration is a lattice. This proves the desired equivalence.
\end{proof}

We immediately obtain the following corollary:

\begin{cor}\label{cor:lat_type}
Assume that the potential $f$ is decreasing and logarithmically convex on $(0,1/2]$. Let $X \subset \mathbb T^d$ be an $N$-point integration lattice, and let $Y \subset \mathbb T^d$ be an $N$-point configuration with the same permutation type as $X$. Then $$ E_F (Y) \ge E_F (X).$$
\end{cor}

In other words, lattices minimize the energy among all sets with the same permutation type. Observing that a  small perturbation of a lattice does not change the  permutation type, we arrive directly at local optimality. 

\begin{cor}\label{cor:lattice}
   If a point set $X$ is a lattice and $f$ is decreasing and logarithmically convex on $(0,1/2]$, then $X$ is a local minimum of $E_F$.
\end{cor}

The last two corollaries  exhibit an example of the behavior discussed in the introduction: certain permutation sets (namely, lattices) minimize the energy among large classes of more general (non-permutation) sets. 

Finally we note that the case  $N=3$ can also be addressed by these methods because this case has only two permutation types in each coordinate, and all permutation sets in this case are lattices (up to the symmetries of the torus). Thus a special case of  Corollary \ref{cor:3} when $f$ is logarithmically convex, can be obtained as a consequence of Corollary \ref{cor:lat_type}. 


\section{The cases $N=4$ and $N=6$ in $\mathbb{T}^2$}\label{sec:N4}

In the setting of $\mathbb{T}^2$, the cases $N=4$ and $N=6$ are unique: the only possible lattices are the identity lattice that places all points on $y=x$, and its reflection that places all points on $y=1-x$. However, unlike in the case $N=3$ (which like $N=5$ is covered by Corollary \ref{cor:3}) there are other permutation sets that are not one of these two cases. 

We first observe that the permutation set whose permutation type is the identity cannot strictly minimize the  energy, even among permutation sets. 


\begin{prop} \label{prop:id}
    Let $N \in \bbN$ and let $f: \bbT \rightarrow [0, \infty)$ be a potential. 
    Then, for all permutations $\sigma$ of $\{0, 1, \dots, N-1\}$,
     we have
    $$
    E_F(X_{\id}) \geq E_F(X_\sigma),
    $$
    where $\id$ denotes the identity permutation and $X_\sigma$ is a permutation set with permutation type $\sigma$
     (see \eqref{e.permutation}).
\end{prop}


\begin{proof}
    The energy of $X_\sigma$ takes the form
    $$
    E_F(X_\sigma) = \sum_{\substack{k, \ell = 0 \\ k \neq \ell}}^{N-1} f\left(\frac{k-\ell}N\right) f\left(\frac{\sigma(k)-\sigma(\ell)}N\right).
    $$
    As multisets we have
    \begin{align*}
        & \left\{f\left(\frac{k-\ell}N\right): k, \ell \in \{0, 1, \dots, N-1\}, k \neq \ell\right\} \\
        = & \left\{f\left(\frac{\sigma(k)-\sigma(\ell)}N\right): k, \ell \in \{0, 1, \dots, N-1\}, k \neq \ell\right\},
    \end{align*}
    by the rearrangement inequality \cite{HLP88} the energy $E_F$ is thus maximized if
    $$
    f\left(\frac{k-\ell}N\right) = f\left(\frac{\sigma(k)-\sigma(\ell)}N\right)
    $$
    for all $k, \ell$, which is the case for the identity permutation.
\end{proof}

Thus, for $N=4$ and $6$, candidates for energy minimizers must be something other than a lattice. We shall consider these cases in detail. 

\subsection{The $N=4$ case}

In the case $N=4$, up to symmetries of the torus, there is only one other permutation set besides the identity permutation, namely
\begin{align} \label{opt_4_perm}
    X^{(4)} \coloneqq \{(0,0),(1/4,1/2),(1/2,1/4),(3/4,3/4)\},
\end{align}
which by Proposition \ref{prop:id} minimizes the energy among all $4$-point permutation sets. If $f$ is differentiable on $(0, 1)$, 
one can perform gradient descent to approach a local minimizer. Following off of an observation in \cite{HO16} we claim that with appropriate convexity conditions this will actually yield a global minimizer, and we present a simple condition how to determine it.

\begin{prop}\label{prop:n4}
    Let  $f$ be differentiable on $(0,1)$, as well as decreasing   and logarithmically convex on $(0,1/2]$. Then there exists an $a\in(0,1/2)$ such that the $4$-point configuration $$X_a^{(4)}\coloneqq\{(0,0),(a,1/2),(1/2,a),(1/2+a,1/2+a)\}$$  is a  minimizer of $E_F$. This value of $a$ satisfies \begin{equation}\label{eq:n4_condition}
        f'(a)f(1/2)=f'(1/2-a)f(1/2-a).
    \end{equation}
    If $f$ is strictly logarithmically convex, this value of $a$ is unique.
\end{prop}

We would like to point out that in \cite{HO16}  such a set was obtained as the energy minimizer of $E_F$ with  $f(t) = \frac12 - t + t^2$ (Fig. \ref{fig:four_configs}) using a proof assisted with numerics. Proposition \ref{prop:n4} provides a simple way to find such minimizers for a wide range of potentials and admits a transparent proof, which we now present.

\begin{proof}
    We note that $f$ being logarithmically convex results in the tensor product potential $F$ being convex on $(0,1)\times(0,1)$, see \cite{SS12b}. This means that as a function of $\bx^i$ and $\bx^j$, the expression $F(\bx^i-\bx^j)$ is convex when the order  of the $x$- and $y$-coordinates of the two points is fixed. Since the sum of convex functions is convex, we  conclude that $E_F$ is convex as a function of the coordinates of our three unfixed points (we assume $\bx^0=(0,0)$) when we restrict it to point configurations of a specific permutation type. Viewing the coordinates of our three unfixed points as an element of $[0,1)^6$, we observe that a convex combination of two element in this space that are of the same permutation type will result in another element of the same permutation type. Therefore the set of all configurations of the same permutation type is a convex set and thus any local minimizer of $E_F$ in this set is a minimizer among all point sets of the same permutation type.

\begin{figure}[htp]
\hspace{1.1cm}
						\begin{tikzpicture}
							\begin{axis}[
								plot box ratio = 1 1,
								xmin=-1/21,xmax=22/21,ymin=-1/21,ymax=22/21,
								font=\footnotesize,
								height=0.35\textwidth,
								width=0.35\textwidth,
								xticklabel = \empty,
								yticklabel = \empty,
								ytick style={draw=none},
								xtick style={draw=none},
								]
								\addplot[only marks,black,mark=*,mark size=2pt,mark options={solid}] coordinates {
									(0, 0) (1/4, 1/2) (1/2, 1/4) (3/4, 3/4)
								};
								\addplot[gray,mark options={solid}] coordinates {
									(0, -0.05) (0, 1.05)
								};
								\addplot[gray,mark options={solid}] coordinates {
									(1.0, -0.05) (1.0, 1.05)
								};
								\addplot[gray,mark options={solid}] coordinates {
									(-0.05, 0) (1.05, 0)
								};
								\addplot[gray,mark options={solid}] coordinates {
									(-0.05, 1.0) (1.05, 1.0)
								};
								\addplot[gray,mark options={solid}, dotted] coordinates {
									(0, 1/4) (1, 1/4)
								};
								\addplot[gray,mark options={solid}, dotted] coordinates {
									(0, 1/2) (1, 1/2)
								};
								\addplot[gray,mark options={solid}, dotted] coordinates {
									(1/4, 0) (1/4, 1)
								};
								\addplot[gray,mark options={solid}, dotted] coordinates {
									(1/2, 0) (1/2, 1)
								};
								\addplot[gray,mark options={solid}, dotted] coordinates {
									(0, 3/4) (1, 3/4)
								};
								\addplot[gray,mark options={solid}, dotted] coordinates {
									(3/4, 0) (3/4, 1)
								};
							\end{axis}
						\end{tikzpicture}
						\begin{tikzpicture}
							\begin{axis}[
								plot box ratio = 1 1,
								xmin=-1/21,xmax=22/21,ymin=-1/21,ymax=22/21,
								font=\footnotesize,
								height=0.35\textwidth,
								width=0.35\textwidth,
								xticklabel = \empty,
								yticklabel = \empty,
								ytick style={draw=none},
								xtick style={draw=none},
								]
								\addplot[only marks,black,mark=*,mark size=2pt,mark options={solid}] coordinates {
									(0, 0) (0.226698825758202, 1/2) (1/2, 0.226698825758202) (0.726698825758202, 0.726698825758202)
								};
								\addplot[gray,mark options={solid}] coordinates {
									(0, -0.05) (0, 1.05)
								};
								\addplot[gray,mark options={solid}] coordinates {
									(1.0, -0.05) (1.0, 1.05)
								};
								\addplot[gray,mark options={solid}] coordinates {
									(-0.05, 0) (1.05, 0)
								};
								\addplot[gray,mark options={solid}] coordinates {
									(-0.05, 1.0) (1.05, 1.0)
								};
								\addplot[gray,mark options={solid}, dotted] coordinates {
									(0, 1/4) (1, 1/4)
								};
								\addplot[gray,mark options={solid}, dotted] coordinates {
									(0, 1/2) (1, 1/2)
								};
								\addplot[gray,mark options={solid}, dotted] coordinates {
									(1/4, 0) (1/4, 1)
								};
								\addplot[gray,mark options={solid}, dotted] coordinates {
									(1/2, 0) (1/2, 1)
								};
								\addplot[gray,mark options={solid}, dotted] coordinates {
									(0, 3/4) (1, 3/4)
								};
								\addplot[gray,mark options={solid}, dotted] coordinates {
									(3/4, 0) (3/4, 1)
								};
							\end{axis}
						\end{tikzpicture}
\caption{The best $4$-point permutation set and the energy minimizer $X_a^{(4)}$ for the potential $f(t) = \frac12 - t + t^2$. Here, $a = 0.22669\dots$ is the solution to the cubic equation $2a^3+a-\frac14=0$.}
\label{fig:four_configs}
\end{figure}
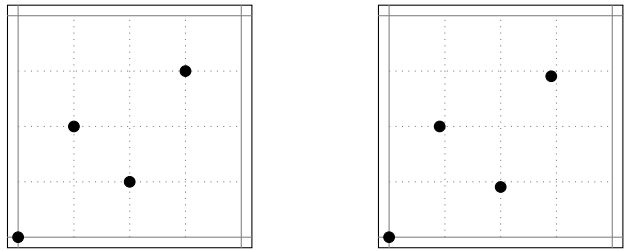

    From Corollary \ref{cor:lat_type} we know that the identity lattice minimizes the energy among sets of the same permutation type, however by Proposition \ref{prop:id} it will not be a (unique) minimizer among all $4$-point configurations in $\bbT^2$. Considering torus symmetries, we need to check only one other permutation type, namely the one of $X_a^{(4)}$. Thus, we only need to show that, for $a$ satisfying \eqref{eq:n4_condition}, $X_a^{(4)}$ is stationary with respect to $E_F$ in order to demonstrate  that it is an energy minimizer. For this, we consider the gradient of the energy functional $E_F(\mathbf 0, \bx^1, \bx^2, \bx^3)$ in terms of the $6$ variables $\bx^1 = (x_1^1, x_2^1)$, $\bx^2=(x_1^2, x_2^2)$, $\bx^3=(x_1^3, x_2^3)$. Taking the partial derivative with respect to any of these $6$ coordinates, we see that it is $0$ if and only if \eqref{eq:n4_condition} holds. We thus conclude that $X_a^{(4)}$ is a minimizer. We know such an $a$ exists because $$E_F(X_a^{(4)})=4f(a)f(1/2)+2f(1/2-a)^2$$ is a convex function in $a$ and thus must have a minimum (the cases of $a=0$ and $a=1/2$ are not minimizers by Corollary \ref{c.noequal}). If $f$ is strictly logarithmically convex then $E_F(X_a^{(4)})$ is a strictly convex function of $a$ and has a unique minimizer, so \eqref{eq:n4_condition} is satisfied by only one value of $a$.
\end{proof}

With this result we reduce the complexity of the problem down to a single variable, making it much more approachable for potentials satisfying the given condition. If we consider the potential given in \eqref{eq:qmc_potential}, for example, we find that the energy minimizer is $X_a^{(4)}$ with $a = a(p)$ satisfying the cubic
$$
\frac{p^2}2a^3+2p\left(1-\frac p{24}\right)a-\frac p2\left(1-\frac p{24}\right) = 0.
$$
It can be noted that for $0 < p \leq 6$, where \eqref{eq:qmc_potential} is logarithmically convex, we have $|a(p)-\frac14| < 0.024$. This phenomenon, that global optimizers do not stray too far from permutation sets, was first observed in \cite{HO16}. Figure \ref{fig:four_configs} shows the global minimizer for the concrete potential of the periodic $L_2$-discrepancy, equivalently $p=6$ above.

\begin{rem}\label{rem:n4not}
We finish this subsection by noting that the permutation set \eqref{opt_4_perm}, i.e. $X_{1/4}^{(4)}$, will only satisfy the condition given in \eqref{eq:n4_condition} if $f(1/2)=f(1/4)$, in which case it will also yield the same energy as the identity permutation. Thus, for potentials that satisfy the conditions of Proposition \ref{prop:n4}, there cannot be a unique energy minimizer that is a permutation set.
\end{rem}


\subsection{The $N=6$ case}

In the $N=6$ case there are many more possible permutation sets, with a total of $9$ cases when taking symmetries into account, see Figure \ref{fig:six_configs}. However, assuming that $f$ is decreasing on $(0,1/2]$ (so that $f(1/6)\ge f(1/3)\ge f(1/2)$) and convex (so that $2f(1/3)\le f(1/6)+f(1/2)$) it can be shown with a direct calculation that, up to symmetries, permutation set with the lowest energy  is $$X^{(6)}=\{(0,0),(1/6,1/3),(1/3,2/3),(1/2,1/6),(2/3,5/6),(5/6,1/2)\},$$ which up to symmetries is the only permutation set that does not have $(1/6,1/6)$ as a distance vector.

We cannot deduce the global optimality of $X_a^{(6)}$ as in Proposition \ref{prop:n4} because unlike the $N=4$ case, we have multiple non-lattice permutation types that remain distinct under reflection, rotation, and translation. However, we can construct  a  configuration which is optimal among all sets with  the same permutation type as $X^{(6)}$.

\begin{prop}\label{prop:n6}
    Assume  $f$ is differentiable on $(0,1)$, as well as   decreasing and logarithmically convex on $(0,1/2]$. Then there exists  $a\in(0,1/4)$ such that the point configuration 
    \begin{align*}X_a^{(6)}\coloneqq\{&(0,0),(a,1/2-a),(2a,1-2a),(1/2,1/2-2a),\\&(a+1/2,1-a),(1/2+2a,1/2)\}\end{align*}  is a  minimizer of $E_F$  among all $6$-point configurations with the same permutation type. This value of $a$ satisfies    \begin{equation}\label{eq:n6_condition}
    f'(a)f(1/2-a)+f'(2a)f(2a)=f'(1/2-a)f(a)+f'(1/2-2a)f(1/2).    
    \end{equation}
    If $f$ is strictly logarithmically convex, this value of $a$ is unique.
    
\end{prop}

\begin{figure}[H]
						\hspace{-0.85cm} \begin{tikzpicture}
							\begin{axis}[
								plot box ratio = 1 1,
								xmin=-1/21,xmax=22/21,ymin=-1/21,ymax=22/21,
								font=\footnotesize,
								height=0.33\textwidth,
								width=0.33\textwidth,
								xticklabel = \empty,
								yticklabel = \empty,
								ytick style={draw=none},
								xtick style={draw=none},
								]
								\addplot[only marks,black,mark=*,mark size=2pt,mark options={solid}] coordinates {
									(0, 0) (1/6, 1/6) (1/3, 1/3) (1/2, 1/2) (2/3, 2/3) (5/6, 5/6)
								};
								\addplot[gray,mark options={solid}] coordinates {
									(0, -0.05) (0, 1.05)
								};
								\addplot[gray,mark options={solid}] coordinates {
									(1.0, -0.05) (1.0, 1.05)
								};
								\addplot[gray,mark options={solid}] coordinates {
									(-0.05, 0) (1.05, 0)
								};
								\addplot[gray,mark options={solid}] coordinates {
									(-0.05, 1.0) (1.05, 1.0)
								};
								\addplot[gray,mark options={solid}, dotted] coordinates {
									(0, 1/6) (1, 1/6)
								};
								\addplot[gray,mark options={solid}, dotted] coordinates {
									(0, 1/3) (1, 1/3)
								};
								\addplot[gray,mark options={solid}, dotted] coordinates {
									(0, 1/2) (1, 1/2)
								};
								\addplot[gray,mark options={solid}, dotted] coordinates {
									(1/6, 0) (1/6, 1)
								};
								\addplot[gray,mark options={solid}, dotted] coordinates {
									(1/3, 0) (1/3, 1)
								};
								\addplot[gray,mark options={solid}, dotted] coordinates {
									(1/2, 0) (1/2, 1)
								};
								\addplot[gray,mark options={solid}, dotted] coordinates {
									(0, 2/3) (1, 2/3)
								};
								\addplot[gray,mark options={solid}, dotted] coordinates {
									(0, 5/6) (1, 5/6)
								};
								\addplot[gray,mark options={solid}, dotted] coordinates {
									(2/3, 0) (2/3, 1)
								};
								\addplot[gray,mark options={solid}, dotted] coordinates {
									(5/6, 0) (5/6, 1)
								};
							\end{axis}
						\end{tikzpicture}
                        \begin{tikzpicture}
							\begin{axis}[
								plot box ratio = 1 1,
								xmin=-1/21,xmax=22/21,ymin=-1/21,ymax=22/21,
								font=\footnotesize,
								height=0.33\textwidth,
								width=0.33\textwidth,
								xticklabel = \empty,
								yticklabel = \empty,
								ytick style={draw=none},
								xtick style={draw=none},
								]
								\addplot[only marks,black,mark=*,mark size=2pt,mark options={solid}] coordinates {
									(0, 0) (1/6, 1/6) (1/3, 1/3) (1/2, 1/2) (2/3, 5/6) (5/6, 2/3)
								};
								\addplot[gray,mark options={solid}] coordinates {
									(0, -0.05) (0, 1.05)
								};
								\addplot[gray,mark options={solid}] coordinates {
									(1.0, -0.05) (1.0, 1.05)
								};
								\addplot[gray,mark options={solid}] coordinates {
									(-0.05, 0) (1.05, 0)
								};
								\addplot[gray,mark options={solid}] coordinates {
									(-0.05, 1.0) (1.05, 1.0)
								};
								\addplot[gray,mark options={solid}, dotted] coordinates {
									(0, 1/6) (1, 1/6)
								};
								\addplot[gray,mark options={solid}, dotted] coordinates {
									(0, 1/3) (1, 1/3)
								};
								\addplot[gray,mark options={solid}, dotted] coordinates {
									(0, 1/2) (1, 1/2)
								};
								\addplot[gray,mark options={solid}, dotted] coordinates {
									(1/6, 0) (1/6, 1)
								};
								\addplot[gray,mark options={solid}, dotted] coordinates {
									(1/3, 0) (1/3, 1)
								};
								\addplot[gray,mark options={solid}, dotted] coordinates {
									(1/2, 0) (1/2, 1)
								};
								\addplot[gray,mark options={solid}, dotted] coordinates {
									(0, 2/3) (1, 2/3)
								};
								\addplot[gray,mark options={solid}, dotted] coordinates {
									(0, 5/6) (1, 5/6)
								};
								\addplot[gray,mark options={solid}, dotted] coordinates {
									(2/3, 0) (2/3, 1)
								};
								\addplot[gray,mark options={solid}, dotted] coordinates {
									(5/6, 0) (5/6, 1)
								};
							\end{axis}
						\end{tikzpicture}
						\begin{tikzpicture}
							\begin{axis}[
								plot box ratio = 1 1,
								xmin=-1/21,xmax=22/21,ymin=-1/21,ymax=22/21,
								font=\footnotesize,
								height=0.33\textwidth,
								width=0.33\textwidth,
								xticklabel = \empty,
								yticklabel = \empty,
								ytick style={draw=none},
								xtick style={draw=none},
								]
								\addplot[only marks,black,mark=*,mark size=2pt,mark options={solid}] coordinates {
									(0, 0) (1/6, 1/6) (1/3, 1/3) (1/2, 5/6) (2/3, 2/3) (5/6, 1/2)
								};
								\addplot[gray,mark options={solid}] coordinates {
									(0, -0.05) (0, 1.05)
								};
								\addplot[gray,mark options={solid}] coordinates {
									(1.0, -0.05) (1.0, 1.05)
								};
								\addplot[gray,mark options={solid}] coordinates {
									(-0.05, 0) (1.05, 0)
								};
								\addplot[gray,mark options={solid}] coordinates {
									(-0.05, 1.0) (1.05, 1.0)
								};
								\addplot[gray,mark options={solid}, dotted] coordinates {
									(0, 1/6) (1, 1/6)
								};
								\addplot[gray,mark options={solid}, dotted] coordinates {
									(0, 1/3) (1, 1/3)
								};
								\addplot[gray,mark options={solid}, dotted] coordinates {
									(0, 1/2) (1, 1/2)
								};
								\addplot[gray,mark options={solid}, dotted] coordinates {
									(1/6, 0) (1/6, 1)
								};
								\addplot[gray,mark options={solid}, dotted] coordinates {
									(1/3, 0) (1/3, 1)
								};
								\addplot[gray,mark options={solid}, dotted] coordinates {
									(1/2, 0) (1/2, 1)
								};
								\addplot[gray,mark options={solid}, dotted] coordinates {
									(0, 2/3) (1, 2/3)
								};
								\addplot[gray,mark options={solid}, dotted] coordinates {
									(0, 5/6) (1, 5/6)
								};
								\addplot[gray,mark options={solid}, dotted] coordinates {
									(2/3, 0) (2/3, 1)
								};
								\addplot[gray,mark options={solid}, dotted] coordinates {
									(5/6, 0) (5/6, 1)
								};
							\end{axis}
						\end{tikzpicture}

                        \hspace{-0.85cm} \begin{tikzpicture}
							\begin{axis}[
								plot box ratio = 1 1,
								xmin=-1/21,xmax=22/21,ymin=-1/21,ymax=22/21,
								font=\footnotesize,
								height=0.33\textwidth,
								width=0.33\textwidth,
								xticklabel = \empty,
								yticklabel = \empty,
								ytick style={draw=none},
								xtick style={draw=none},
								]
								\addplot[only marks,black,mark=*,mark size=2pt,mark options={solid}] coordinates {
									(0, 0) (1/6, 1/6) (1/3, 1/3) (1/2, 5/6) (2/3, 1/2) (5/6, 2/3)
								};
								\addplot[gray,mark options={solid}] coordinates {
									(0, -0.05) (0, 1.05)
								};
								\addplot[gray,mark options={solid}] coordinates {
									(1.0, -0.05) (1.0, 1.05)
								};
								\addplot[gray,mark options={solid}] coordinates {
									(-0.05, 0) (1.05, 0)
								};
								\addplot[gray,mark options={solid}] coordinates {
									(-0.05, 1.0) (1.05, 1.0)
								};
								\addplot[gray,mark options={solid}, dotted] coordinates {
									(0, 1/6) (1, 1/6)
								};
								\addplot[gray,mark options={solid}, dotted] coordinates {
									(0, 1/3) (1, 1/3)
								};
								\addplot[gray,mark options={solid}, dotted] coordinates {
									(0, 1/2) (1, 1/2)
								};
								\addplot[gray,mark options={solid}, dotted] coordinates {
									(1/6, 0) (1/6, 1)
								};
								\addplot[gray,mark options={solid}, dotted] coordinates {
									(1/3, 0) (1/3, 1)
								};
								\addplot[gray,mark options={solid}, dotted] coordinates {
									(1/2, 0) (1/2, 1)
								};
								\addplot[gray,mark options={solid}, dotted] coordinates {
									(0, 2/3) (1, 2/3)
								};
								\addplot[gray,mark options={solid}, dotted] coordinates {
									(0, 5/6) (1, 5/6)
								};
								\addplot[gray,mark options={solid}, dotted] coordinates {
									(2/3, 0) (2/3, 1)
								};
								\addplot[gray,mark options={solid}, dotted] coordinates {
									(5/6, 0) (5/6, 1)
								};
							\end{axis}
						\end{tikzpicture}
                        \begin{tikzpicture}
							\begin{axis}[
								plot box ratio = 1 1,
								xmin=-1/21,xmax=22/21,ymin=-1/21,ymax=22/21,
								font=\footnotesize,
								height=0.33\textwidth,
								width=0.33\textwidth,
								xticklabel = \empty,
								yticklabel = \empty,
								ytick style={draw=none},
								xtick style={draw=none},
								]
								\addplot[only marks,black,mark=*,mark size=2pt,mark options={solid}] coordinates {
									(0, 0) (1/6, 1/3) (1/3, 1/6) (1/2, 1/2) (2/3, 5/6) (5/6, 2/3)
								};
								\addplot[gray,mark options={solid}] coordinates {
									(0, -0.05) (0, 1.05)
								};
								\addplot[gray,mark options={solid}] coordinates {
									(1.0, -0.05) (1.0, 1.05)
								};
								\addplot[gray,mark options={solid}] coordinates {
									(-0.05, 0) (1.05, 0)
								};
								\addplot[gray,mark options={solid}] coordinates {
									(-0.05, 1.0) (1.05, 1.0)
								};
								\addplot[gray,mark options={solid}, dotted] coordinates {
									(0, 1/6) (1, 1/6)
								};
								\addplot[gray,mark options={solid}, dotted] coordinates {
									(0, 1/3) (1, 1/3)
								};
								\addplot[gray,mark options={solid}, dotted] coordinates {
									(0, 1/2) (1, 1/2)
								};
								\addplot[gray,mark options={solid}, dotted] coordinates {
									(1/6, 0) (1/6, 1)
								};
								\addplot[gray,mark options={solid}, dotted] coordinates {
									(1/3, 0) (1/3, 1)
								};
								\addplot[gray,mark options={solid}, dotted] coordinates {
									(1/2, 0) (1/2, 1)
								};
								\addplot[gray,mark options={solid}, dotted] coordinates {
									(0, 2/3) (1, 2/3)
								};
								\addplot[gray,mark options={solid}, dotted] coordinates {
									(0, 5/6) (1, 5/6)
								};
								\addplot[gray,mark options={solid}, dotted] coordinates {
									(2/3, 0) (2/3, 1)
								};
								\addplot[gray,mark options={solid}, dotted] coordinates {
									(5/6, 0) (5/6, 1)
								};
							\end{axis}
						\end{tikzpicture}
                        \begin{tikzpicture}
							\begin{axis}[
								plot box ratio = 1 1,
								xmin=-1/21,xmax=22/21,ymin=-1/21,ymax=22/21,
								font=\footnotesize,
								height=0.33\textwidth,
								width=0.33\textwidth,
								xticklabel = \empty,
								yticklabel = \empty,
								ytick style={draw=none},
								xtick style={draw=none},
								]
								\addplot[only marks,black,mark=*,mark size=2pt,mark options={solid}] coordinates {
									(0, 0) (1/6, 1/2) (1/3, 2/3) (1/2, 1/6) (2/3, 1/3) (5/6, 5/6)
								};
								\addplot[gray,mark options={solid}] coordinates {
									(0, -0.05) (0, 1.05)
								};
								\addplot[gray,mark options={solid}] coordinates {
									(1.0, -0.05) (1.0, 1.05)
								};
								\addplot[gray,mark options={solid}] coordinates {
									(-0.05, 0) (1.05, 0)
								};
								\addplot[gray,mark options={solid}] coordinates {
									(-0.05, 1.0) (1.05, 1.0)
								};
								\addplot[gray,mark options={solid}, dotted] coordinates {
									(0, 1/6) (1, 1/6)
								};
								\addplot[gray,mark options={solid}, dotted] coordinates {
									(0, 1/3) (1, 1/3)
								};
								\addplot[gray,mark options={solid}, dotted] coordinates {
									(0, 1/2) (1, 1/2)
								};
								\addplot[gray,mark options={solid}, dotted] coordinates {
									(1/6, 0) (1/6, 1)
								};
								\addplot[gray,mark options={solid}, dotted] coordinates {
									(1/3, 0) (1/3, 1)
								};
								\addplot[gray,mark options={solid}, dotted] coordinates {
									(1/2, 0) (1/2, 1)
								};
								\addplot[gray,mark options={solid}, dotted] coordinates {
									(0, 2/3) (1, 2/3)
								};
								\addplot[gray,mark options={solid}, dotted] coordinates {
									(0, 5/6) (1, 5/6)
								};
								\addplot[gray,mark options={solid}, dotted] coordinates {
									(2/3, 0) (2/3, 1)
								};
								\addplot[gray,mark options={solid}, dotted] coordinates {
									(5/6, 0) (5/6, 1)
								};
							\end{axis}
						\end{tikzpicture}

                        \hspace{-0.85cm} \begin{tikzpicture}
							\begin{axis}[
								plot box ratio = 1 1,
								xmin=-1/21,xmax=22/21,ymin=-1/21,ymax=22/21,
								font=\footnotesize,
								height=0.33\textwidth,
								width=0.33\textwidth,
								xticklabel = \empty,
								yticklabel = \empty,
								ytick style={draw=none},
								xtick style={draw=none},
								]
								\addplot[only marks,black,mark=*,mark size=2pt,mark options={solid}] coordinates {
									(0, 0) (1/6, 1/6) (1/3, 5/6) (1/2, 1/2) (2/3, 1/3) (5/6, 2/3)
								};
								\addplot[gray,mark options={solid}] coordinates {
									(0, -0.05) (0, 1.05)
								};
								\addplot[gray,mark options={solid}] coordinates {
									(1.0, -0.05) (1.0, 1.05)
								};
								\addplot[gray,mark options={solid}] coordinates {
									(-0.05, 0) (1.05, 0)
								};
								\addplot[gray,mark options={solid}] coordinates {
									(-0.05, 1.0) (1.05, 1.0)
								};
								\addplot[gray,mark options={solid}, dotted] coordinates {
									(0, 1/6) (1, 1/6)
								};
								\addplot[gray,mark options={solid}, dotted] coordinates {
									(0, 1/3) (1, 1/3)
								};
								\addplot[gray,mark options={solid}, dotted] coordinates {
									(0, 1/2) (1, 1/2)
								};
								\addplot[gray,mark options={solid}, dotted] coordinates {
									(1/6, 0) (1/6, 1)
								};
								\addplot[gray,mark options={solid}, dotted] coordinates {
									(1/3, 0) (1/3, 1)
								};
								\addplot[gray,mark options={solid}, dotted] coordinates {
									(1/2, 0) (1/2, 1)
								};
								\addplot[gray,mark options={solid}, dotted] coordinates {
									(0, 2/3) (1, 2/3)
								};
								\addplot[gray,mark options={solid}, dotted] coordinates {
									(0, 5/6) (1, 5/6)
								};
								\addplot[gray,mark options={solid}, dotted] coordinates {
									(2/3, 0) (2/3, 1)
								};
								\addplot[gray,mark options={solid}, dotted] coordinates {
									(5/6, 0) (5/6, 1)
								};
							\end{axis}
						\end{tikzpicture}
                        \begin{tikzpicture}
							\begin{axis}[
								plot box ratio = 1 1,
								xmin=-1/21,xmax=22/21,ymin=-1/21,ymax=22/21,
								font=\footnotesize,
								height=0.33\textwidth,
								width=0.33\textwidth,
								xticklabel = \empty,
								yticklabel = \empty,
								ytick style={draw=none},
								xtick style={draw=none},
								]
								\addplot[only marks,black,mark=*,mark size=2pt,mark options={solid}] coordinates {
									(0, 0) (1/6, 1/6) (1/3, 2/3) (1/2, 1/3) (2/3, 5/6) (5/6, 1/2)
								};
								\addplot[gray,mark options={solid}] coordinates {
									(0, -0.05) (0, 1.05)
								};
								\addplot[gray,mark options={solid}] coordinates {
									(1.0, -0.05) (1.0, 1.05)
								};
								\addplot[gray,mark options={solid}] coordinates {
									(-0.05, 0) (1.05, 0)
								};
								\addplot[gray,mark options={solid}] coordinates {
									(-0.05, 1.0) (1.05, 1.0)
								};
								\addplot[gray,mark options={solid}, dotted] coordinates {
									(0, 1/6) (1, 1/6)
								};
								\addplot[gray,mark options={solid}, dotted] coordinates {
									(0, 1/3) (1, 1/3)
								};
								\addplot[gray,mark options={solid}, dotted] coordinates {
									(0, 1/2) (1, 1/2)
								};
								\addplot[gray,mark options={solid}, dotted] coordinates {
									(1/6, 0) (1/6, 1)
								};
								\addplot[gray,mark options={solid}, dotted] coordinates {
									(1/3, 0) (1/3, 1)
								};
								\addplot[gray,mark options={solid}, dotted] coordinates {
									(1/2, 0) (1/2, 1)
								};
								\addplot[gray,mark options={solid}, dotted] coordinates {
									(0, 2/3) (1, 2/3)
								};
								\addplot[gray,mark options={solid}, dotted] coordinates {
									(0, 5/6) (1, 5/6)
								};
								\addplot[gray,mark options={solid}, dotted] coordinates {
									(2/3, 0) (2/3, 1)
								};
								\addplot[gray,mark options={solid}, dotted] coordinates {
									(5/6, 0) (5/6, 1)
								};
							\end{axis}
						\end{tikzpicture}
                        \begin{tikzpicture}
							\begin{axis}[
								plot box ratio = 1 1,
								xmin=-1/21,xmax=22/21,ymin=-1/21,ymax=22/21,
								font=\footnotesize,
								height=0.33\textwidth,
								width=0.33\textwidth,
								xticklabel = \empty,
								yticklabel = \empty,
								ytick style={draw=none},
								xtick style={draw=none},
								]
								\addplot[only marks,black,mark=*,mark size=2pt,mark options={solid}] coordinates {
									(0, 0) (1/6, 1/3) (1/3, 2/3) (1/2, 1/6) (2/3, 5/6) (5/6, 1/2)
								};
								\addplot[gray,mark options={solid}] coordinates {
									(0, -0.05) (0, 1.05)
								};
								\addplot[gray,mark options={solid}] coordinates {
									(1.0, -0.05) (1.0, 1.05)
								};
								\addplot[gray,mark options={solid}] coordinates {
									(-0.05, 0) (1.05, 0)
								};
								\addplot[gray,mark options={solid}] coordinates {
									(-0.05, 1.0) (1.05, 1.0)
								};
								\addplot[gray,mark options={solid}, dotted] coordinates {
									(0, 1/6) (1, 1/6)
								};
								\addplot[gray,mark options={solid}, dotted] coordinates {
									(0, 1/3) (1, 1/3)
								};
								\addplot[gray,mark options={solid}, dotted] coordinates {
									(0, 1/2) (1, 1/2)
								};
								\addplot[gray,mark options={solid}, dotted] coordinates {
									(1/6, 0) (1/6, 1)
								};
								\addplot[gray,mark options={solid}, dotted] coordinates {
									(1/3, 0) (1/3, 1)
								};
								\addplot[gray,mark options={solid}, dotted] coordinates {
									(1/2, 0) (1/2, 1)
								};
								\addplot[gray,mark options={solid}, dotted] coordinates {
									(0, 2/3) (1, 2/3)
								};
								\addplot[gray,mark options={solid}, dotted] coordinates {
									(0, 5/6) (1, 5/6)
								};
								\addplot[gray,mark options={solid}, dotted] coordinates {
									(2/3, 0) (2/3, 1)
								};
								\addplot[gray,mark options={solid}, dotted] coordinates {
									(5/6, 0) (5/6, 1)
								};
							\end{axis}
						\end{tikzpicture}
\caption{The $9$ possible distinct permutation sets in approximate order of decreasing energy, with the best possible configuration last.}
\label{fig:six_configs}
\end{figure}
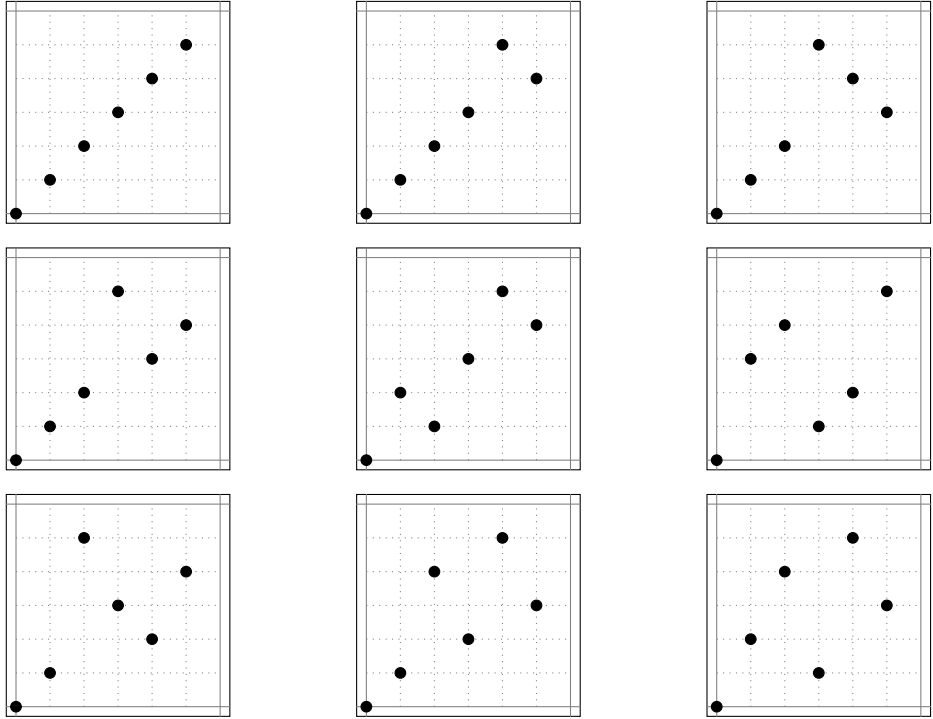

\begin{proof}

Just as in the $N=4$ case we note that since $f$ is logarithmically convex and differentiable, we can find a minimizer given a specific permutation type by finding a configuration of the permutation type that is stationary with respect to the energy $E_F$ in the $x$- and $y$-coordinates of every point. It is straightforward to calculate the $10$ partial derivatives with respect to the coordinates of $\bx^1=(x^1_1, x^1_2), \dots, \bx^5=(x_1^5, x_2^5)$. Setting these partial derivatives to $0$, one obtains condition \eqref{eq:n6_condition} for $X_a^{(6)}$.
We use the same argument as in the proof of Proposition \ref{prop:n4} to claim that a minimizer exists, and that it is unique in the case that $f$ is strictly logarithmically convex. 
\end{proof}

Unlike  the case $N=4$,  condition \eqref{eq:n6_condition} can be conceivably satisfied by the permutation set $X^{(6)}_{1/6}$ and in this case the condition becomes $$f'(1/6)(f(1/3)-f(1/2))=f'(1/3)(f(1/6)-f(1/3)).$$
Nevertheless, for the majority of potentials this equality will not be satisfied and thus in general the energy minimizer for $N=6$ will not be a permutation set, see Figure \ref{fig:6_pts}.






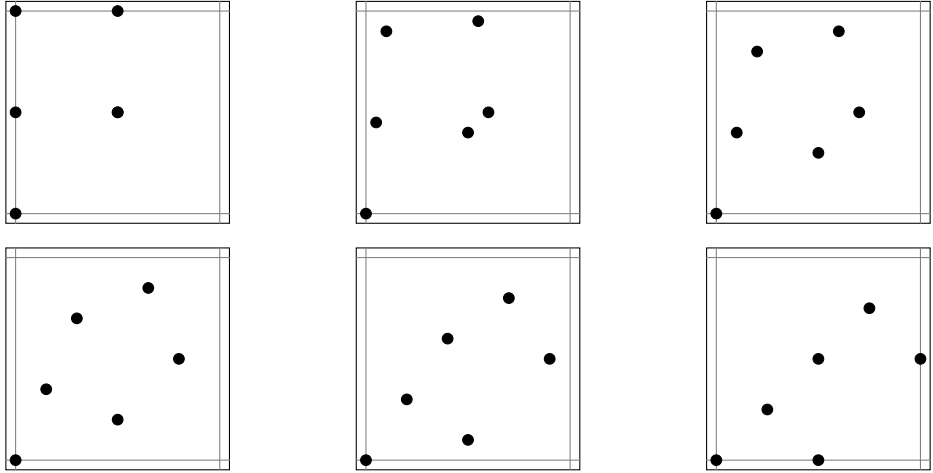
\begin{figure}[htp]
\hspace{-1cm}
						\begin{tikzpicture}
							\begin{axis}[
								plot box ratio = 1 1,
								xmin=-1/21,xmax=22/21,ymin=-1/21,ymax=22/21,
								font=\footnotesize,
								height=0.33\textwidth,
								width=0.33\textwidth,
								xticklabel = \empty,
								yticklabel = \empty,
								ytick style={draw=none},
								xtick style={draw=none},
								]
								\addplot[only marks,black,mark=*,mark size=2pt,mark options={solid}] coordinates {
									(0, 0) (0, 0.5) (0, 1) (0.5, 0.5) (0.5, 1) (0.5, 0.5)
								};
								\addplot[gray,mark options={solid}] coordinates {
									(0, -0.05) (0, 1.05)
								};
								\addplot[gray,mark options={solid}] coordinates {
									(1.0, -0.05) (1.0, 1.05)
								};
								\addplot[gray,mark options={solid}] coordinates {
									(-0.05, 0) (1.05, 0)
								};
								\addplot[gray,mark options={solid}] coordinates {
									(-0.05, 1.0) (1.05, 1.0)
								};
							\end{axis}
						\end{tikzpicture}
                        \begin{tikzpicture}
							\begin{axis}[
								plot box ratio = 1 1,
								xmin=-1/21,xmax=22/21,ymin=-1/21,ymax=22/21,
								font=\footnotesize,
								height=0.33\textwidth,
								width=0.33\textwidth,
								xticklabel = \empty,
								yticklabel = \empty,
								ytick style={draw=none},
								xtick style={draw=none},
								]
								\addplot[only marks,black,mark=*,mark size=2pt,mark options={solid}] coordinates {
									(0, 0) (0.05, 0.45) (0.1, 0.9) (0.5, 0.4) (0.55, 0.95) (0.6, 0.5)
								};
								\addplot[gray,mark options={solid}] coordinates {
									(0, -0.05) (0, 1.05)
								};
								\addplot[gray,mark options={solid}] coordinates {
									(1.0, -0.05) (1.0, 1.05)
								};
								\addplot[gray,mark options={solid}] coordinates {
									(-0.05, 0) (1.05, 0)
								};
								\addplot[gray,mark options={solid}] coordinates {
									(-0.05, 1.0) (1.05, 1.0)
								};
							\end{axis}
						\end{tikzpicture}
                        \begin{tikzpicture}
							\begin{axis}[
								plot box ratio = 1 1,
								xmin=-1/21,xmax=22/21,ymin=-1/21,ymax=22/21,
								font=\footnotesize,
								height=0.33\textwidth,
								width=0.33\textwidth,
								xticklabel = \empty,
								yticklabel = \empty,
								ytick style={draw=none},
								xtick style={draw=none},
								]
								\addplot[only marks,black,mark=*,mark size=2pt,mark options={solid}] coordinates {
									(0, 0) (0.1, 0.4) (0.2, 0.8) (0.5, 0.3) (0.6, 0.9) (0.7, 0.5)
								};
								\addplot[gray,mark options={solid}] coordinates {
									(0, -0.05) (0, 1.05)
								};
								\addplot[gray,mark options={solid}] coordinates {
									(1.0, -0.05) (1.0, 1.05)
								};
								\addplot[gray,mark options={solid}] coordinates {
									(-0.05, 0) (1.05, 0)
								};
								\addplot[gray,mark options={solid}] coordinates {
									(-0.05, 1.0) (1.05, 1.0)
								};
							\end{axis}
						\end{tikzpicture}

\hspace{-1cm}
                        \begin{tikzpicture}
							\begin{axis}[
								plot box ratio = 1 1,
								xmin=-1/21,xmax=22/21,ymin=-1/21,ymax=22/21,
								font=\footnotesize,
								height=0.33\textwidth,
								width=0.33\textwidth,
								xticklabel = \empty,
								yticklabel = \empty,
								ytick style={draw=none},
								xtick style={draw=none},
								]
								\addplot[only marks,black,mark=*,mark size=2pt,mark options={solid}] coordinates {
									(0, 0) (0.15, 0.35) (0.3, 0.7) (0.5, 0.2) (0.65, 0.85) (0.8, 0.5)
								};
								\addplot[gray,mark options={solid}] coordinates {
									(0, -0.05) (0, 1.05)
								};
								\addplot[gray,mark options={solid}] coordinates {
									(1.0, -0.05) (1.0, 1.05)
								};
								\addplot[gray,mark options={solid}] coordinates {
									(-0.05, 0) (1.05, 0)
								};
								\addplot[gray,mark options={solid}] coordinates {
									(-0.05, 1.0) (1.05, 1.0)
								};
							\end{axis}
						\end{tikzpicture}
                        \begin{tikzpicture}
							\begin{axis}[
								plot box ratio = 1 1,
								xmin=-1/21,xmax=22/21,ymin=-1/21,ymax=22/21,
								font=\footnotesize,
								height=0.33\textwidth,
								width=0.33\textwidth,
								xticklabel = \empty,
								yticklabel = \empty,
								ytick style={draw=none},
								xtick style={draw=none},
								]
								\addplot[only marks,black,mark=*,mark size=2pt,mark options={solid}] coordinates {
									(0, 0) (0.2, 0.3) (0.4, 0.6) (0.5, 0.1) (0.7, 0.8) (0.9, 0.5)
								};
								\addplot[gray,mark options={solid}] coordinates {
									(0, -0.05) (0, 1.05)
								};
								\addplot[gray,mark options={solid}] coordinates {
									(1.0, -0.05) (1.0, 1.05)
								};
								\addplot[gray,mark options={solid}] coordinates {
									(-0.05, 0) (1.05, 0)
								};
								\addplot[gray,mark options={solid}] coordinates {
									(-0.05, 1.0) (1.05, 1.0)
								};
							\end{axis}
						\end{tikzpicture}
                        \begin{tikzpicture}
							\begin{axis}[
								plot box ratio = 1 1,
								xmin=-1/21,xmax=22/21,ymin=-1/21,ymax=22/21,
								font=\footnotesize,
								height=0.33\textwidth,
								width=0.33\textwidth,
								xticklabel = \empty,
								yticklabel = \empty,
								ytick style={draw=none},
								xtick style={draw=none},
								]
								\addplot[only marks,black,mark=*,mark size=2pt,mark options={solid}] coordinates {
									(0, 0) (0.25, 0.25) (0.5, 0.5) (0.5, 0) (0.75, 0.75) (1, 0.5)
								};
								\addplot[gray,mark options={solid}] coordinates {
									(0, -0.05) (0, 1.05)
								};
								\addplot[gray,mark options={solid}] coordinates {
									(1.0, -0.05) (1.0, 1.05)
								};
								\addplot[gray,mark options={solid}] coordinates {
									(-0.05, 0) (1.05, 0)
								};
								\addplot[gray,mark options={solid}] coordinates {
									(-0.05, 1.0) (1.05, 1.0)
								};
							\end{axis}
						\end{tikzpicture}
\caption{$X_a^{(6)}$ for $a \in \{0, 0.05, 0.1, 0.15, 0.2, 0.25\}$ (with the double points $(0.5, 0.5)$ and $(0, 0) \simeq (0, 1)$ for $a=0$). For the potential $f(t) = \frac12-t+t^2$ the minimizer is attained for $a=0.16375\dots$ solving $20a^3-15a^2+\frac{13}2a-\frac34=0$.}\label{fig:6_pts}
\end{figure}

\section*{Acknowledgment}


DB was supported by the Simons Travel Support for Mathematicians and the Fulbright-NAWI Graz Visiting Professor fellowship.   NN was supported by the ESF Plus Young Researchers Group ReSIDA-H2 (SAB, 10064931). We would also like to thank Alexey Glazyrin for fruitful discussions. 

\appendix

\section*{Appendix}

\section{A collection of potentials}\label{s.potentials}

In this section we  collect some potentials which are relevant in the context of tensor product energy. Natural examples arise from discrepancy theory, quasi-Monte Carlo integration and harmonic analysis. F
We also give  other examples which might shed  light on optimal point configurations. 

\paragraph{Quasi-Monte Carlo integration.} Let $s \in \bbN$ and let $a_{-1}, a_0, a_1, \dots,$ $a_s \geq 0$ be parameters with $a_s > 0$ and $a_{-1}+ a_0 > 0$. Let $H^s_\ba(\bbT)$ be the Sobolev space of periodic functions $f: \bbT \rightarrow \bbR$ such that $f^{(k)} \in L_2(\bbT)$ for all weak derivatives of order $k=0, 1, \dots, s$, equipped with the norm (which is a norm by the assumptions on the parameters $a_i$)
$$
\|f\|_{H^s_\ba(\bbT)}^2 \coloneqq a_{-1} \left(\int_\bbT f(x) \, \dx x\right)^2 + \sum_{k=0}^s a_k \int_\bbT f^{(k)}(x)^2 \, \dx x.
$$
This turns $H^s_\ba(\bbT)$ into a reproducing kernel Hilbert space with kernel given by
$$
K_\ba(x, y) = \sum_{m \in \bbZ} \frac{\exp(2\pi\iu m (x-y))}{a_{-1} \delta_{m, 0} + \sum_{k=0}^s a_k (2\pi m)^{2k}}.
$$
For $d$ dimensions we simply take a $d$-fold tensor product $H^s_\ba(\bbT^d) = H^s_\ba(\bbT) \otimes \dots \otimes H^s_\ba(\bbT)$, whose kernel can be given by
$$
K_\ba^d(\bx, \by) = \prod_{i=1}^d K_\ba(x_i, y_i).
$$
Spaces like these are sometimes also called \emph{Korobov spaces} \cite{Dic08, DSWW06, GL22, Kuo03}, although note that some authors use this term for different spaces \cite{DTU18, KNU25, Kor59}. For $X \sbse \bbT^d, \#X = N$ the worst case error of the quasi-Monte Carlo rule can be determined to be \cite{NW10}
$$
\sup_{\|f\|_{H^s_\ba(\bbT^d)} \leq 1} \left|\frac1N \sum_{\bx \in X} f(\bx) - \int_{\bbT^d} f(\bx) \, \dx\bx\right|^2 = -\frac1{(a_{-1} + a_0)^d} + \frac1{N^2} \sum_{\bx, \by \in X} K_\ba^d(\bx, \by).
$$
Minimizing the worst case error is thus equivalent to minimizing
$$
\sum_{\substack{\bx, \by \in X \\ \bx \neq \by}} K_\ba^d(\bx, \by),
$$
which, since the kernel is of the form $K_\ba^d(\bx, \by) = \prod_{i=1}^d f_\ba(x_i-y_i)$, can be interpreted as a tensor product energy. A particular example would be, for $p > 0$ a parameter, $a_{-1} = 1$, $a_{s} = p^{-1}$ and $a_0 = a_1 = \dots = a_{s-1} = 0$, for which the kernel is of the form
$$
K_\ba(x, y) = 1 + (-1)^{s-1} \frac{p}{(2s)!} B_{2s}(x-y),
$$
where $B_{2s}$ is a Bernoulli polynomial of degree $2s$, given by
\begin{align} \label{eq:Bernoulli_poly}
	B_{2s}(t) = - (2s)! \sum_{m \neq 0} \frac{\exp(2\pi\iu m t)}{(2\pi\iu m)^{2s}}.
\end{align}
The case of $s = 1$ and $p \in \{1, 6\}$ was considered in \cite{HO16}. More generally, these potentials were studied for $s=1$ in \cite{Nag25b}. Similar but slightly different potentials were considered in \cite{BD13b}.

\paragraph{(Smooth) Discrepancy.} Discrepancy quantifies how uniform a set $X \sbse \bbT^d$ is in a geometric way. We describe the definition here and show its relation to the quasi-Monte Carlo method. Let $s \in \bbN$ and for $x \in \bbR$ define the B-splines
$$
B_{0, 1}^1(x) \coloneqq \chi^{}_{[0, 1)}(x)
$$
and recursively
$$
B_{0, 1}^s(x) \coloneqq \int_{x-1}^x B_{0, 1}^{s-1}(u) \, \dx u = \int_0^1 B_{0, 1}^{s-1}(x-u) \, \dx u.
$$
Note that $\supp B_{0, 1}^s = [0, s]$. With $(x)_+ \coloneqq \max\{0, x\}$ we even have explicitly
$$
B_{0, 1}^s(x) = \frac{1}{(s-1)!} \sum_{k=0}^s (-1)^k {s \choose k} (x-k)_+^{s-1}.
$$
For $0 \leq \alpha < 1$ and $0 < \delta \leq 1$ set
$$
B_{\alpha, \delta}^s(x) \coloneqq B_{0, 1}^s\left(\frac{x-\alpha}\delta\right) = \delta^{-1} \int_0^\delta B_{\alpha, \delta}^{s-1}(x-u) \, \dx u
$$
with $\supp B_{\alpha, \delta}^s = [\alpha, \alpha + s \delta]$. For $x \in [0, 1) \simeq \bbT$ introduce the periodization
$$
\tilde B_{\alpha, \delta}^s(x) \coloneqq \sum_{k \in \bbZ} B_{\alpha, \delta}^s(k+x) = \delta^{-1} \int_0^\delta \tilde B_{\alpha, \delta}^{s-1}(x-u) \, \dx u,
$$
where $x-u \in \bbT$ with wrap-around on the torus. Finally, let
\begin{align} \label{eq:phi_recursive}
	\varphi_{\alpha, \delta}^s(x) \coloneqq \delta^{s-1} \tilde B_{\alpha, \delta}^s(x) = \int_0^\delta \varphi_{\alpha, \delta}^{s-1}(x-u) \, \dx u.
\end{align}
It is easy to verify that
\begin{itemize}
	\item $\varphi_{\alpha, \delta}^1(x) = \chi^{}_{[\alpha, \{\alpha + \delta\})}(x)$ in the sense of periodic intervals \cite{HKP21}, where $\{\cdot\}$ denotes the fractional part,
	
	\item $\varphi_{\alpha, 1}^s(x) \equiv 1$ and
	
	\item $\int_0^1 \varphi_{\alpha, \delta}^s(x) \, \dx x = \delta^s$.
\end{itemize}
Let us collect some facts about these functions $\varphi_{\alpha, \delta}^s$.
\begin{lem} \label{lem:spline_properties}
	We have the Fourier representation (with convergence in $L_2$)
	\begin{align} \label{eq:spline_fourier}
		\varphi_{\alpha, \delta}^s(x) = \delta^s + \sum_{m \neq 0} \left(\frac{1-\exp(-2\pi\iu m \delta)}{2\pi\iu m}\right)^s \exp(2\pi\iu m (x-\alpha)).
	\end{align}
	Furthermore, for all $x, y \in [0, 1)$ it holds
	\begin{align} \label{eq:spline_inner_prod}
		\int_0^1 \int_0^1 \varphi_{\alpha, \delta}^s(x) \varphi_{\alpha, \delta}^s(y) \, \dx \delta \, \dx \alpha = \frac1{2s+1} + \frac{(-1)^{s-1}}{(s!)^2} B_{2s}(|x-y|).
	\end{align}
\end{lem}
\begin{proof}
	We prove \eqref{eq:spline_fourier} via induction on $s$. By periodicity we may assume $\alpha = 0$. For $s = 1$ we have $\varphi_{0, \delta}^1(x) = \chi^{}_{[0, \delta)}(x)$ with Fourier coefficients
	$$
	\int_0^1 \varphi_{0, \delta}^1(x) \exp(-2\pi\iu m x) \, \dx x = \begin{cases}
		\delta, & m = 0, \\
		\frac{1-\exp(-2\pi\iu m \delta)}{2\pi\iu m}, & m \neq 0.
	\end{cases}
	$$
	Assuming the formula holds for $s-1$, we have (using periodicity)
	\begin{align*}
		& \int_0^1 \varphi_{0, \delta}^s(x) \exp(-2\pi\iu m x) \, \dx x = \int_0^\delta \int_0^1 \varphi_{0, \delta}^{s-1}(x-u) \exp(-2\pi\iu m x) \, \dx x \, \dx u \\
		= & \int_0^\delta \exp(-2\pi\iu m u) \, \dx u \int_0^1 \varphi_{0, \delta}^{s-1}(x) \exp(-2\pi\iu m x) \, \dx x \\
		= & \begin{cases}
			\delta^s, & m = 0, \\
			\left(\frac{1-\exp(-2\pi\iu m \delta)}{2\pi\iu m}\right)^s, & m \neq 0.
		\end{cases}
	\end{align*}
	This shows \eqref{eq:spline_fourier}. We use this Fourier representation to show the second formula \eqref{eq:spline_inner_prod}. Indeed, we can then write
	\begin{align*}
		& \int_0^1 \int_0^1 \varphi_{\alpha, \delta}^s(x) \varphi_{\alpha, \delta}^s(y) \, \dx \delta \, \dx \alpha \\
		= & \frac1{2s+1} + \sum_{m \neq 0} \left( \int_0^1 \left|\frac{1-\exp(-2\pi\iu m \delta)}{2\pi\iu m}\right|^{2s} \, \dx \delta\right) \exp(2\pi\iu m (x-y)).
	\end{align*}
	Since for $t \in \bbR$ it holds
	$$
	|1- \exp(-\iu t)|^{2} = -\exp(-\iu t) (1-\exp(\iu t))^2,
	$$
	for $t = 2\pi m \delta$ we have
	\begin{align*}
		\int_0^1 |1- \exp(-2\pi\iu m \delta)|^{2s} \, \dx \delta & = (-1)^s \int_0^1 \exp(-2\pi\iu s m \delta) (1-\exp(2\pi\iu m \delta))^{2s} \, \dx \delta \\
		& = \sum_{k=0}^{2s} (-1)^{s+k} {2s \choose k} \int_0^1 \exp(2\pi\iu (k-s) m \delta) \, \dx \delta \\
		& = {2s \choose s}.
	\end{align*}
	This finally leads to the desired result (using \eqref{eq:Bernoulli_poly})
	\begin{align*}
		\int_0^1 \int_0^1 \varphi_{\alpha, \delta}^s(x) \varphi_{\alpha, \delta}^s(y) \, \dx \delta \, \dx \alpha & = \frac1{2s+1} + {2s \choose s} \sum_{m \neq 0} \frac{\exp(2\pi\iu m (x-y))}{(2\pi m)^{2s}} \\
		& = \frac1{2s+1} + \frac{(-1)^{s-1}}{(s!)^2} B_{2s}(|x-y|).
	\end{align*}
\end{proof}
The \emph{$s$-smooth periodic $L_2$-discrepancy} of a point set $X \sbse \bbT^d, \#X = N$ will be defined as
$$
L_{2, s}^{\text{per}}(X)^2 \coloneqq \int_{[0, 1)^d} \int_{(0, 1]^d} \left|\frac1N \sum_{\bx \in X} \prod_{i=1}^d \varphi_{\alpha_i, \delta_i}^s(x_i) - \int_{\bbT^d}\prod_{i=1}^d \varphi_{\alpha_i, \delta_i}^s(x_i) \, \dx \bx \right|^2 \, \dx \underline{\delta} \, \dx \underline{\alpha},
$$
where $\underline{\alpha} = (\alpha_1, \dots, \alpha_d)$ and $\underline{\delta} = (\delta_1, \dots, \delta_d)$. From Lemma \ref{lem:spline_properties} we conclude the following \emph{Warnock-type formula} for periodic discrepancy (the case for $s=1$ already appearing in \cite{HKP21, HO16}).
\begin{thm} \label{thm:smooth_discr}
	It holds
	\begin{align*}
		& L_{2, s}^{\text{per}}(X)^2 \\
		= & -\frac1{(2s+1)^d} + \frac1{N^2} \sum_{\bx, \by \in X} \prod_{i=1}^d \left(\frac1{2s+1} + \frac{(-1)^{s-1}}{(s!)^2} B_{2s}(|x_i-y_i|)\right).
	\end{align*}
\end{thm}
While the general equivalence of QMC and (smooth) discrepancy is well-known \cite{Tem19}, this explicit formula does not seem to have appeared in the literature so far. Comparing with $K_\ba^d$ from the previous part, we see that this is equivalent to the QMC worst case error over the space $H_\ba^s(\bbT^d)$ with parameters
$$
a_{-1} = 2s+1, a_0 = a_1 = \dots = a_{s-1} = 0, a_s = {2s \choose s}^{-1}.
$$

\paragraph{Diaphony.} Diaphony \cite{Lev95, Zin76} of $X \sbse \bbT^d, \#X = N$ is a quantity given by (without the square root)
$$
\sum_{\bm \in \bbZ} \left|\frac1N \sum_{\bx \in X} \exp(2\pi\iu \bm^\top \bx)\right|^2 \prod_{i=1}^d \begin{cases}
	1, & m_i = 0, \\
	\frac1{2\pi m_i}, & m_i \neq 0.
\end{cases}
$$
This is turns out to be the worst case error for the quasi-Monte Carlo rule over the space $H^1_\ba(\bbT^d)$ with parameters $a_{-1}=1, a_0 = 0, a_{1} = 1$.

\paragraph{Tensorized Poisson kernel.} The \emph{Poisson kernel} with parameter $r \in (0, 1)$ is given by
$$
P_r(t)= \sum_{m \neq 0} r^{|m|} \exp(2\pi\iu m t) = \frac{1-r^2}{1 - 2 r \cos(2\pi t) + r^2}
$$
for $t \in \bbT$. Tensorization leads to the kernel
$$
K^d_{P_r}(\bx, \by) \coloneqq \prod_{i=1}^d P_r(x_i-y_i),
$$
For fixed $r$ this can be seen as a kernel over a Sobolev space of infinite dominating mixed smoothness (with certain exponential decay of the Fourier coefficients).

\paragraph{Logarithmically convex potentials.} We finish by mentioning some logarithmically convex potentials. Here $k$ can be seen as a shape parameter and $c$ is a parameter that controls the logarithmic convexity for a certain range.
\begin{itemize}
	\item Polynomial potential
	$$
	f(t) = 1 + c\left(t^k + (1-t)^k\right), \quad k > 1, 0<c \leq k-1
	$$
	
	\item Another polynomial potential
	$$
	f(t) = 1 + c \left|2t-1\right|^k, \quad k > 1, 0<c \leq k-1
	$$
	
	\item Potential via the entropy function
	$$
	f(t) = 1 + c\left(t \ln t + (1-t) \ln (1-t)\right), \quad 0 < c \leq c_0 = 1.23325675\dots
	$$
	
	\item Ellipse potential
	$$
	f(t) = 1 - c \sqrt{t(1-t)}, \quad 0 < c \leq (3/2)^{3/2} = 1.8371173\dots
	$$
\end{itemize}


\bibliographystyle{plain}
\bibliography{refs}

\begin{thebibliography}{10}

\bibitem{BHS}
Sergiy~V. B., D.~P. Hardin, and E.~B. Saff.
\newblock {\em Discrete Energy on Rectifiable Sets}.
\newblock Springer, 2019.

\bibitem{BD09}
J.~Baldeaux and J.~Dick.
\newblock Q{MC} rules of arbitrary high order: reproducing kernel {H}ilbert
  space approach.
\newblock {\em Constr. Approx.}, 30(3):495--527, 2009.

\bibitem{BT03}
A.~Berlinet and C.~Thomas-Agnan.
\newblock {\em Reproducing Kernel Hilbert Spaces in Probability and
  Statistics}.
\newblock Springer New York, NY, {F}irst edition, 2003.

\bibitem{BTY12-2}
D.~Bilyk, V.~N. Temlyakov, and R.~Yu.
\newblock Fibonacci sets and symmetrization in discrepancy theory.
\newblock {\em Journal of Complexity}, 28(1):18--36, 2012.

\bibitem{BTY12-1}
D.~Bilyk, V.~N Temlyakov, and R.~Yu.
\newblock The ${L}_2$ discrepancy of two-dimensional lattices.
\newblock In {\em Recent Advances in Harmonic Analysis and Applications: in
  Honor of Konstantin Oskolkov}, pages 63--77. Springer, 2012.

\bibitem{BHS12}
J.~Brauchart, {D. P.} Hardin, and {E. B.} Saff.
\newblock Discrete energy asymptotics on a {R}iemannian circle.
\newblock {\em Uniform Distribution Theory}, 7(2):77--108, 2012.

\bibitem{BD13b}
J.~S. Brauchart and J.~A. Dick.
\newblock A characterization of {S}obolev spaces on the sphere and an extension
  of {S}tolarsky’s invariance principle to arbitrary smoothness.
\newblock {\em Constr Approx}, 38:397--445, 2013.

\bibitem{CFP}
J.~A. Carrillo, A.~Figalli, and F.~S. Patacchini.
\newblock Geometry of minimizers for the interaction energy with mildly
  repulsive potentials.
\newblock {\em Ann. Inst. H. Poincar\'e{} C Anal. Non Lin\'eaire},
  34(5):1299--1308, 2017.

\bibitem{CS02}
W.~W.~L. Chen and M.~M. Skriganov.
\newblock Explicit constructions in the classical mean squares problem in
  irregularities of point distribution.
\newblock {\em Journal für die reine und angewandte Mathematik},
  2002(545):67--95, 2002.

\bibitem{CK07}
H.~Cohn and A.~Kumar.
\newblock Universally optimal distribution of points on spheres.
\newblock {\em J Amer Math Soc}, 20:99--148, 2007.

\bibitem{CKMRV22}
H.~Cohn, A.~Kumar, S.~D. Miller, D.~Radchenko, and M.~Viazovska.
\newblock Universal optimality of the {$E_8$} and {L}eech lattices and
  interpolation formulas.
\newblock {\em Ann. of Math. (2)}, 196(3):983--1082, 2022.

\bibitem{CKM16}
H.~Cohn, A.~Kumar, and G.~Minton.
\newblock Optimal simplices and codes in projective spaces.
\newblock {\em Geom. Topol.}, 20(3):1289--1357, 2016.

\bibitem{Dic08}
J.~Dick.
\newblock {Walsh spaces containing smooth functions and quasi–Monte Carlo
  rules of arbitrary high order}.
\newblock {\em SIAM Journal on Numerical Analysis}, 46(3):1519--1553, 2008.

\bibitem{DP10}
J.~Dick and F.~Pillichshammer.
\newblock {\em Digital Nets and Sequences: Discrepancy Theory and Quasi–Monte
  Carlo Integration}.
\newblock Cambridge University Press, 2010.

\bibitem{DSWW06}
J.~Dick, I.~H. Sloan, X.~Wang, and H.~Wo\'zniakowski.
\newblock Good lattice rules in weighted {K}orobov spaces with general weights.
\newblock {\em Numerische Mathematik}, 103:63--97, 2006.

\bibitem{DTU18}
D.~D{\~u}ng, V.~Temlyakov, and T.~Ullrich.
\newblock {\em Hyperbolic cross approximation}.
\newblock Springer, 2018.

\bibitem{Fej56}
L.~Fejes~T\'oth.
\newblock On the sum of distances determined by a pointset.
\newblock {\em Acta Math. Acad. Sci. Hungar.}, 7:397--401, 1956.

\bibitem{GL22}
T.~Goda and P.~L'Ecuyer.
\newblock Construction-free median quasi-{M}onte {C}arlo rules for function
  spaces with unspecified smoothness and general weights.
\newblock {\em SIAM Journal on Scientific Computing}, 44(4):A2765--A2788, 2022.

\bibitem{GSY16a}
T.~Goda, K.~Suzuki, and T.~Yoshiki.
\newblock An explicit construction of optimal order quasi--{M}onte {C}arlo
  rules for smooth integrands.
\newblock {\em SIAM Journal on Numerical Analysis}, 54(4):2664--2683, 2016.

\bibitem{GQ10}
B.-N. Guo and F.~Qi.
\newblock A property of logarithmically absolutely monotonic functions and the
  logarithmically complete monotonicity of a power-exponential function.
\newblock {\em Politehn. Univ. Bucharest Sci. Bull. Ser. A Appl. Math. Phys.},
  72(2):21--30, 2010.

\bibitem{HLP88}
G.~H. Hardy, J.~E. Littlewood, and G.~P\'olya.
\newblock {\em Inequalities}.
\newblock Cambridge Mathematical Library. Cambridge University Press,
  Cambridge, 1988.

\bibitem{HKP21}
A.~Hinrichs, R.~Kritzinger, and F.~Pillichshammer.
\newblock Extreme and periodic ${L}_2$ discrepancy of plane point sets.
\newblock {\em Acta Arithmetica}, 199:163--198, 2021.

\bibitem{HO16}
A.~Hinrichs and J.~Oettershagen.
\newblock {Optimal point sets for quasi-Monte Carlo integration of bivariate
  periodic functions with bounded mixed derivatives}.
\newblock In Ronald Cools and Dirk Nuyens, editors, {\em Monte Carlo and
  Quasi-Monte Carlo Methods: MCQMC, Leuven, Belgium, April 2014}, pages
  385--405. Springer International Publishing, 2016.

\bibitem{KNU25}
B.~K\"a{\ss}emodel, N.~Nagel, and T.~Ullrich.
\newblock Tent transformed order 2 nets and quasi-{M}onte {C}arlo rules with
  quadratic error decay.
\newblock {\em arXiv:2505.10955}, 2025.

\bibitem{Kor59}
N.~M. Korobov.
\newblock Approximate evaluation of repeated integrals.
\newblock {\em Dokl. Akad. Nauk SSSR}, 124:1207--1210, 1959.

\bibitem{Kuo03}
F.~Y. Kuo.
\newblock Component-by-component constructions achieve the optimal rate of
  convergence for multivariate integration in weighted {K}orobov and {S}obolev
  spaces.
\newblock {\em Journal of Complexity}, 19(3):301--320, 2003.
\newblock Oberwolfach Special Issue.

\bibitem{Leh85}
D.~H. Lehmer.
\newblock Interesting series involving the central binomial coefficient.
\newblock {\em The American Mathematical Monthly}, 92(7):449--457, 1985.

\bibitem{Lev95}
V.~F. Lev.
\newblock On two versions of ${L}_2$-discrepancy and geometrical interpretation
  of diaphony.
\newblock {\em Acta Mathematica Hungarica}, 64(4):281--300, 1995.

\bibitem{Nag25b}
N.~Nagel.
\newblock On the global optimality of {F}ibonacci lattices in the torus.
\newblock {\em arXiv:2502.17082}, 2025.

\bibitem{NW10}
E.~Novak and H.~Wo\'zniakowski.
\newblock {\em Tractability of multivariate problems. {V}olume {II}: {S}tandard
  information for functionals}, volume~12 of {\em EMS Tracts in Mathematics}.
\newblock European Mathematical Society (EMS), Z\"urich, 2010.

\bibitem{SS12b}
V.~Shikhman and O.~Stein.
\newblock On jet-convex functions and their tensor products.
\newblock {\em Optimization}, 61(6):717--731, 2012.

\bibitem{Tem19}
V.~Temlyakov.
\newblock Connections between numerical integration, discrepancy, dispersion,
  and universal discretization.
\newblock {\em The SMAI Journal of computational mathematics}, 5:185--209,
  2019.

\bibitem{War72}
T.~T. Warnock.
\newblock Computational investigations of low-discrepancy point sets.
\newblock In S.~K. Zaremba, editor, {\em Applications of Number Theory to
  Numerical Analysis}, pages 319--343. Academic Press, 1972.

\bibitem{Zen77}
A.~A. Zhensykbaev.
\newblock The best quadrature formula for some classes of periodic
  differentiable functions.
\newblock {\em Izv. Akad. Nauk SSSR Ser. Mat.}, 41(5):1110--1124, 1200, 1977.

\bibitem{Zin76}
P.~Zinterhof.
\newblock {\"Uber einige Absch\"atzungen bei der Approximation von Funktionen
  mit Gleichverteiltheitsmethoden}.
\newblock {\em \"Osterr Akad Wiss Math-Naturwiss Kl}, 185(2):121--132, 1976.

\end{thebibliography}

\end{document}